\documentclass[journal]{IEEEtran}

\usepackage{graphicx}
\usepackage{tikz}
\usepackage{color}
\usepackage{amsmath, amsfonts, amssymb, mathrsfs}
\usepackage[amsmath,thmmarks]{ntheorem}
\usepackage{subfigure}
\usepackage{float}  
\usepackage{hyperref} 
\usepackage{epstopdf}
\usepackage{pgfplots}
\usepackage{arydshln}
\usepackage{mathabx}
\usepackage{nth}
\usepackage{cite}
\usepackage{bm,upgreek}
\usepackage{fixmath}
\usepackage{algorithm}
\usepackage{algorithmic}
\usepackage{bbm}
\usepackage[draft]{todonotes} 

\theoremstyle{plain}
\theoremheaderfont{\itshape\bfseries}
\theorembodyfont{\normalfont\itshape}
\theoremseparator{:}
\theoremsymbol{}
\newtheorem{theorem}{Theorem}
\newtheorem{lemma}[theorem]{Lemma}

\newtheorem{proposition}[theorem]{Proposition}

\newtheorem{assum}{Assumption}
\newtheorem{definition}{Definition}

\theoremstyle{nonumberplain}

\theoremstyle{plain}
\theoremheaderfont{\itshape\bfseries}
\theorembodyfont{\normalfont\itshape}
\theoremseparator{:}
\theoremsymbol{}
\newtheorem{remark}{Remark}

\theoremstyle{plain}
\theoremheaderfont{\itshape\bfseries}
\theorembodyfont{\normalfont}
\theoremseparator{:}

\theoremstyle{nonumberplain}
\theoremsymbol{\rule{1ex}{1ex}}

\newtheorem{proof}{Proof}
\newlength\fheight
\newlength\fwidth
\newcommand{\abs}[1]{\left| #1 \right|}
\newcommand{\norm}[1]{\| #1 \|}
\newcommand{\inn}[2]{\langle #1,#2 \rangle}
\newcommand{\lrbrace}[1]{\left\{ #1 \right\}}

\newcommand{\normi}[1]{{\left\vert\kern-0.25ex\left\vert\kern-0.25ex\left\vert #1 
		\right\vert\kern-0.25ex\right\vert\kern-0.25ex\right\vert}}
\newcommand{\inni}[2]{{\langle\kern-0.25ex\langle #1,#2
		\rangle\kern-0.25ex\rangle}}

\def\X{\mathcal{X}}

\def\T{ \mathrm{T} }								

\def\nat{ \mathbb{N} }								

\def\real{ \mathbb{R} }								

\def\Erw{ \mathbb{E} }

\def\Prob{ \mathbb{P} }								

\def\V{\mathcal{V}}

\def\X{\mathcal{X}}

\def\rmh{\mathrm{h}}

\def\bx{ \mathbold{x} }	
\def\bv{ \mathbf{v} }
\def\bg{ \mathbf{g} }

\def\ru{ \mathrm{u} }
\def\bM{ \mathbold{M} }
\def\bz{ \mathbold{z} }
\def\bfX{ \mathbf{X} }
\def\bfY{ \mathbf{Y} }
\def\bfh{ \mathbf{h} }
\def\boldp{ \mathbold{p} }
\def\boldlambda{ \mathbold{\lambda} }
\def\boldR{ \mathbold{R} }
\def\boldS{ \mathbold{S} }
\def\calA{ \mathcal{A} }
\def\bfA{ \mathbf{A} }
\def\bfQ{ \mathbf{Q} }
\def\bfC{ \mathbf{C} }

\def\boldX{ \mathbold{X} }
\def\boldZ{ \mathbold{Z} }
\def\boldY{ \mathbold{Y} }
\def\bfPi{ \mathbf{\Pi} }
\def\boldLambda{ \mathbold{\Lambda} }
\def\boldGamma{ \mathbold{\Gamma} }
\def\rmF{ \mathrm{F} }	
\def\bfPhi{ \mathbf{\Phi} }								
\def\by{  \mathbold{y} }									

\DeclareMathOperator{\VI}{\textnormal{VI}}

\DeclareMathOperator{\SOL}{\textnormal{SOL}}
\DeclareMathOperator{\GNE}{\text{GNE}}
\DeclareMathOperator{\CVio}{\textnormal{CVio}}

\DeclareMathOperator*{\argmax}{arg\,max}

\newcommand\restr[2]{{
		\left.\kern-\nulldelimiterspace 
		#1 
		\right|_{#2} 
}}

\def\Z{\mathcal{Z}}

\title{Coordinated Online Learning for Multi-Agent Systems with Coupled Constraints and Perturbed Utility Observations}
\author{Ezra~Tampubolon,~\IEEEmembership{Student~Member,~IEEE.} and Holger~Boche,~\IEEEmembership{Fellow,~IEEE,} 
	\thanks{Holger Boche and Ezra~Tampubolon are with the Technische Universit\"at M\"unchen,
		Lehrstuhl f\"ur Theoretische Informationstechnik, Germany. E-mail: boche@tum.de, ezra.tampubolon@tum.de. Additionally, H.Boche is with the CASA – Cyber Security in the Age of Large-Scale Adversaries – Excellenzcluster,
		Ruhr Universit\"at Bochum, Bochum, Germany.
		
		E. Tampubolon was supported by the Deutsche Forschungsgemeinschaft (DFG,
		German Research Foundation) under Grant Bo 1734/22-1 and Grant Bo 1734/24-1. H. Boche was supported by the Deutsche Forschungsgemeinschaft (DFG, German Research Foundation) under Germany’s Excellence Strategy - EXC 2092 CASA - 390781972 and in part by the DFG within the Gottfried Wilhelm Leibniz Prize under Grant BO 1734/20-1.
}}

\begin{document}

\maketitle

\begin{abstract}
Competitive non-cooperative online decision-making agents whose actions increase congestion of scarce resources constitute a model for widespread modern large-scale applications. To ensure sustainable resource behavior, we introduce a novel method to steer the agents toward a stable population state, fulfilling the given coupled resource constraints. The proposed method is a decentralized resource pricing method based on the resource loads resulting from the augmentation of the game’s Lagrangian. Assuming that the online learning agents have only noisy first-order utility feedback, we show that for a polynomially decaying agents’ step size/learning rate, the population’s dynamic will almost surely converge to generalized Nash equilibrium. A particular consequence of the latter is the fulfillment of resource constraints in the asymptotic limit. Moreover, we investigate the finite-time quality of the proposed algorithm by giving a non-asymptotic time decaying bound for the expected amount of resource constraint violation.
\end{abstract}


\section{Introduction}
In a vast number of real-world applications, such as smart grid \cite{Mohsenian2010,Saad2012,Li2016,Ma2013,Parise2014,Ma2016,Grammatico20162}, competitive market \cite{Li2015}, and network management \cite{Barrera2015}, the modeling of system participants as competitive selfish rational agents has become popular and has led to fruitful discussions about system design. Likewise, such modeling is suitable for large-scale real-time systems, occurring in world-changing technologies, such as IoT, 5G, and the smart industry. In fact, the competitiveness between the participants in such systems arises from the lack of communication and agreement due to the inherent complexity and the requirements on both the stringent latency and the high-flexibility. 
\paragraph*{Game Theory} The concept of game theory is suitable for describing, analyzing, and forecasting the systems above's behavior. The reason is that this concept considers a set of entities/agents that behave rationally by maximizing their yield and, therefore, closely resemble the participants of a smart large-scale system.
Besides, the game theory assumes that the agents are non-cooperative in the sense that one agent's strategy is not visible to the others so that the aspects mentioned in the previous paragraph are respected.
Moreover, it models the aspect of openness of modern systems by assuming that one agent's yield depends not only on her action but also on others' actions. 

One of the main goals of game theory is to predict, in a repeated setting and for an individual optimizing behavior of the agents, the long-term behavior of the population \cite{Sandholm2012,Newton2018}. In this direction, the central sub-concept is the so-called \textit{Nash equilibrium (NE)} \cite{Nash1951}, which denotes a stable state where no agent is interested in deviating from her strategy. 

\paragraph*{Online Learning in a Competitive Environment} In a repeated game, a single rational agent faces sequential decision-making in an unknown environment. A reasonable assumption from the perspective of machine learning is that she applies the so-called \textit{no-regret policy} (see, e.g., \cite{Shalev-Shwartz2012,Bub2012,Belmaga2018}). This policy asserts that the agent endeavors to choose actions such that the \textit{regret}, i.e., the cumulative difference between the instantaneous yields and the corresponding highest possible yields, grows slower than the number of rounds.
 The canonical class of no-regret policies in an unknown environment is the so-called \textit{online mirror descent (OMD)}. 
 
 The OMD is the online extension of the celebrated mirror descent algorithm \cite{Nesterov2009}. Correspondingly, OMD consists of the gradient step, providing the direction of the steepest decay/increase of the loss/payoff, and the "mirror" step, mapping the latter back to the region of the feasible actions. One advantage of the OMD is that it constitutes a generalization of the online version of the projected gradient descent and provides therefore a richer model for the online decision-making process. Another advantage of OMD is that its performance has a weak influence on the dimension of the underlying decision space provided that the mirror step is chosen appropriately (see e.g., \cite{Nemirovski2008}).

\paragraph*{Resource Constraints}
In widespread practical applications, the action of the agents causes the utilization of specific limited resources. For example, in network applications, the user's (agents) choice of data transfer paths (strategy) increases the congestion of specific links and routers (resources) with limited capacity.
Also, a similar scenario is observable in electric mobility \cite{Ma2013}, fog networking (see e.g. \cite{Chen2018}), and wireless communication \cite{Scutari2012}.

In applications where agents' action is coupled with the resources' congestion, the system designer and manager have to deal with the danger of resource overload due to agents' egoistic behavior. The latter is because the state of resources' overutilization can cause immense degradation of the overall system performance (see, e.g., the problem of congestion collapse in networked system \cite{Abbas2016}), and also, negative environmental issues (e.g., caused by high CO2 emissions of electrical energy driven resources). Another example of events justifying the importance resource sustainability aspect in a system of egoistic optimizing agents is the flash crash in US financial markets due to fully automated computerized trading (see, e.g., \cite{Borch2016}).      

\paragraph*{Problem Description}
This work addresses the problem of controlling egoistic online learning agents, so that in the long-term, the population's action converges to a stable state fulfilling the resource constraints.
A challenge associated with this issue is to design a decentralized congestion control method that does not provide direct commands to each agent by a centralized instance. Moreover, it is desired that the control method demands as little information about agents' characteristics as possible. The reason for requiring those properties is that methods contrary to those would need, e.g., if the number of agents is massive, exceptionally high computational power for the information processing and policy generation. Furthermore, such methods would be inflexible for the possible exit and entrance of new agents and, therefore, unsuitable for modern systems such as IoT.

\paragraph*{Our Contributions}  
Our main contribution is a novel method that solves the above control problem.
 Its core is a resource pricing method, aiming to give both, incentives (rather than direct commands) to all agents for acting sustainably, and stability for the population's state. Our pricing method requires only the current congestion state of the resources and not the agents' specific characteristics. Moreover, it is done by the resources themselves, rather than by a centralized instance.  

Assuming for simplicity that the feedback noise is persistent, we give the following results ($n$ denotes time):
\begin{itemize}
\item 
If the agents' step size $\gamma_{n}$ is of order $\Theta(n^{-p})$ where $p\in (1/2,1]$, we ensure
the almost sure fulfillment of the resource constraints in the asymptotic limit by showing
the almost sure convergence of the population iterate to a (variational) NE of the corresponding game underlying coupled resource constraints.
\item For the case where the agents' step size sequence is of order $\gamma_{n}=\Theta(n^{-1/2})$, we provide a time-decaying non-asymptotic upper bound of order $\mathcal{O}(\ln^{3/2}(n)/\sqrt{n})$ for the expected resource constraints' violation caused by the ergodic average of the population's action.
\item We show that the ergodic average of the population's iterate almost surely fulfills the resource constraints in the asymptotic limit for a large class of decaying step size sequences of order $\gamma_{n}=\Theta(n^{-p})$, where $p\in(0,1]$.
\end{itemize}

\subsection*{Relation to Prior Works}

\paragraph{Learning in Games}
Our work is related to the works investigating the dynamic of agents in a competitive setting. In particular, among them are those closely related to ours, which generate long-term results with different agent types. The latter includes no-learning agents, e.g., greedy agents with best-response dynamics, and learning agents, e.g., the fictitious playing, the gradient playing, and the farsighted reinforcement learning.
For a comprehensive review of the literature on those topics, we refer to \cite{Fudenberg1998,Zhang2019}.

 In this work, we focus on online learning agents applying the canonical mirror descent algorithm \cite{Nemirovski2008,Nesterov2009}. Therefore, the closest work to ours is \cite{Mertikopoulos2018} (along with several extensions such as \cite{DuvocellaMertikopoulos2018,ZhouMertikopoulos2018}). However, our work is in contrast to \cite{Mertikopoulos2018}, since we aim to control competitive online learning agents respective to coupled resource constraints, rather than to predict the long-term outcome. For this reason, the admissible set of population strategy profiles is not necessarily of product structure. Consequently, we have to modify the decentralized algorithm given in \cite{Mertikopoulos2018} and make use of coordinators to handle such inter-agent constraints. Another contrast to \cite{Mertikopoulos2018}, is that our technique relies directly on the Martingale Convergence Theorem, while in \cite{Mertikopoulos2018}, involved methods using continuous interpolation is used to show the convergence of the given discrete-time algorithm.

\paragraph{Generalized Nash equilibrium and Coupled Constraints} As we consider a non-cooperative game with coupled constraints having the generalized Nash equilibrium (GNE) as the central concept, we mention in the following some related works on this topic. Of interest is the subclass of the variational Nash equilibrium (VNE), defined as the solution of the well-known concept of the variational inequality (VI) \cite{Facchinei1}. Several characterizations of the Nash equilibrium of games with coupled constraints have been made in the works \cite{Facchinei1,Facchinei2007,Pavel2007,Yin2011,Kulkarni2012,Arslan2012}, leveraging from the Lagrangian duality theory. Our work does not overlap with these works, as our emphasis is not on analyzing the GNE. In this regard, other works worth mentioning are those investigating the deviation of GNE from the population's welfare, i.e., the efficiency of GNE.
On the efficiency of NE, there is extensive literature (e.g., \cite{Johari2004,Roughgarden2004,Papadimitriou2001}). 

However, relevant to our work is the very recent analysis given in \cite{Kulkarni2019}. The point there essential for us to is the hint that the VNE might be efficient, while the GNE might be arbitrarily inefficient. So a method converging to a VNE might not only support resource sustainability but also increase the population's welfare. In our numerical simulation, we also observe this effect.

\paragraph{Nash Equilibrium Findings}Besides controlling online learning agents, our method is also suitable for the Nash equilibrium finding. There is a large body of literature considering this problem. Reviewing all of them is beyond the scope of this work. Thus we concentrate on those that consider a similar setting as ours, i.e., the game with coupled constraints. 
	Most of the existing works in this direction, such as \cite{Yin2011,Paccagnan2017,Grammatico2017}, propose a primal-dual algorithm based on the fixed-point methods for finding the solution of a variational inequality (see, e.g., Chapter 12 in \cite{Facchinei1}), resulting in a Euclidean-projection based algorithm. 
In contrast, our method uses the mirror map, which constitutes a generalization of the Euclidean projection. For this reason, the performance of our method has a much weaker dimension dependence. 
However, a problem with utilizing the mirror map is that we cannot use the usual convergence proof via a fixed-point approach for variational inequality.

Other essential differences between our algorithm and the first-order algorithms for finding GNE of a game with coupled constraints are that they mostly use a constant step size, do not consider the possibility of noise in the feedback information, and do not have a non-asymptotic guarantee of the violation of the resource constraints.

Also worth mentioning are works that consider the payoff-based approach (see, e.g., \cite{Marden2009,Tatarenko2019}), where each agent can only observe its obtained payoffs. Such an approach is important for some applications (see e.g., \cite{Marden2013,Zhu2013}). In contrast, our work assumes that the agents have a gradient observation of their utility. Nevertheless, investigating a method for gradient feedback constitutes a cornerstone for a payoff-based approach. Therefore, we expect that from our work, one can generate a control-method via pricing for resource sustainability for online learning payoff-based agents. 

\paragraph{Resource Congestion Control} 
Lastly, we mention that the problem of alleviation of resource congestion can also be combined with other objectives, such as maximizing the population's welfare. This practice is common, e.g., in network/internet congestion control \cite{Srikant2004,Low1999} and control theory \cite{Paganini2005}. However, most of the algorithms for fulfilling such an extended task require a high degree of control of the agents or specific information about the agents (e.g., their utilities). In contrast, our method only assumes that the agents are online learners. Moreover, it is based on a pricing instrument using only the resource congestion state. However, the cost we pay is that we cannot guarantee (theoretically) performance gains other than the alleviation of resource congestion. Nevertheless, we can see in simulation, an additional gain of the population's wealth (see Section \ref{Sec:NumExp}).    
\subsection*{Paper Organization}
 The structure of our work is as follows. In Section \ref{Sec:Model}, we set up the multi-agent setting of our consideration by introducing the underlying game and the notion of coupled (resource) constraints (Subsection \ref{Subsec:Game}). Besides, we introduce in this section our agents' learning model (Subsection \ref{Subsec:aaksjsjshsshsgggshssss}) which corresponds to the descent method based on the concept of mirror map (Subsection \ref{Subsec:aiaissjshhsjshsss}).
 
 In Section \ref{Sec:aiishhhsggdgdgdhhsjjshhssss}, we propose the desired pricing algorithm that leads the online learning agent to a resource sustainable state. Therein, we also provide intuitions leading to the design of this method.
 
 The remaining sections are devoted to the analysis of our proposed method structured as follows. In Section \ref{Sec:aooaahhsjsshhddjjddd}, we show that our method ensures resource sustainability in the asymptotic region by showing the convergence of the population's iterate to a stable set satisfying coupled constraints.

  In Section \ref{Sec:aiaiishshggsshgshsss}, we quantify the extent of coupled constraint violations reduction achievable by our method. Specifically, we derive a time-decaying bound for the constraint violation caused by the price-controlled online learning population. In Section \ref{Sec:iaaiiaahsgshgshhsssss}, we close the gap between parameter choices provided in Sections \ref{Sec:aooaahhsjsshhddjjddd} and \ref{Sec:aiaiishshggsshgshsss} and derive the convergence of the ergodic average of the population's iterate to the stable state of the interest. Lastly, the final section (Section \ref{Sec:NumExp}) is devoted to practical simulations. There, we not only support our theoretical findings, but also show that our method is not too conservative, since it can also ensure the increase in the population's wealth. Furthermore, we compare our method with state of the art methods and show that ours may outperform them.

\subsection*{Basic Notations}
For $n\in\nat_{0}$, $[n]$ (resp. $[n]_{0}$) denotes the set of integers between $1$ (resp. $0$) and $n$. We use boldface letters to distinguish between vectors (or sequences) and scalars, and between vector-valued and scalar-valued functions. Upright letters stand for functions and matrices. We write a random variable by capital letter, but a capital upright letter stands for matrices. Given a vector/sequence $\bx$, $\bx_{r}$ denotes the $r$-th member of $\bx$. We use the same notation for the vector-valued function and the random vector/sequence. In case we have a sequence of vectors $\boldLambda$, we denote $\boldLambda_{n}^{r}$ as the $r$-th entry of the $n$-th member.

In this work, we always consider the usual \textit{Euclidean space} $\V:=(\real^{D},\norm{\cdot})$ equipped with a \textit{norm} $\norm{\cdot}$. For a convex subset $\calA\subseteq\real^{D}$, $\text{relint}(A)$ denotes the \textit{relative interior} of $\calA$, and $\norm{A}$ denotes the \textit{diameter} of $\calA$, i.e., the supremum of $\norm{\bx-\by}$ over all $\bx,\by\in \calA$. The \textit{Euclidean projection} onto a closed convex subset $\calA$ of $\real^{D}$ is denoted by $\bfPi_{A}$. The \textit{dual norm} of $\norm{\cdot}$ is denoted by $\norm{\cdot}_{*}$. $\V^{*}=(\real^{D},\norm{\cdot}_{*})$ denotes the \textit{dual space} of the Euclidean normed space $\V$. 

$\bg:\real^{D}\rightarrow\real^{D}$ is said to be \textit{Lipschitz continuous} on an a non-empty subset $\Z\subset (\real^{D},\norm{\cdot})$ with constant $L>0$ if $\norm{\bg(\bx)-\bg(\bz)}_{*}\leq L\norm{\bx-\bz}$, $\forall \bx,\bz\in\mathcal{Z}$. $\bg$ is said to be \textit{monotone} on $\Z$ if $\inn{\bx_{1}-\bx_{2}}{\bg(\bx_{1})-\bg(\bx_{2})}\leq 0$, for all $\bx_{1},\bx_{2}\in \mathcal{Z}$. If in the latter, strict inequality holds for $\bx_{1}\neq \bx_{2}$, then $\bg$ is said to be \textit{strictly monotone}. 

 In this work we assume that a \textit{probability space} $(\Omega,\Sigma,\Prob)$, and a \textit{filtration} $\mathbb{F}:=(\mathcal{F}_{n})_{n\in\nat_{0}}$ therein. By the abbreviation a.s. we mean almost sure w.r.t. this probability space. 
 For ease of notation, we denote for each $n\in\nat_{0}$, the conditional expectation $\Erw[\cdot|\mathcal{F}_{n}]$ given $\mathcal{F}_{n}$, simply by $\Erw_{n}[\cdot]$. Let $(\bM_{n})_{n}$ be a sequence of random variables taking values on a normed Euclidean space $(\real^{D},\norm{\cdot})$. We say $(\bM_{n})_{n}$ is a \textit{(resp. super-,sub-)martingale} if $(\bM_{n})_{n}$ is \textit{adapted}, i.e., $\bM_{n}$ is $\mathcal{F}_{n}$-measureable for all $n$, $(\bM_{n})_{n}$ is \textit{integrable}, i.e., $\Erw[\norm{\bM_{n}}]<\infty$ for all $n$, and  $\Erw_{n}[\bM_{n+1}]=\bM_{n}$ (resp. $\leq$, $\geq$, instead of $=$) for all $n$.
\section{Model Description and Preliminaries}
\label{Sec:Model}
In this section, we formalize the framework of our consideration, i.e., the setting of agents competing for scarce resources. Explicitly, we introduce the notion of non-cooperative game underlying coupled resource constraints. Moreover, we introduce the so-called mirror map, which provides a model of how a selfish online learning agent realizes her decision from the first-order feedback. 
\subsection{Continuous Games with Coupled Constraints}
\label{Subsec:Game}
\paragraph*{Continuous Non-Cooperative Game}Throughout this work, we consider a finite set $[N]$ of agents playing a (repeated) \emph{non-cooperative game (NG)} $\Gamma$. During the NG, every agent $i\in[N]$ chooses and applies an action/strategy $\bx^{(i)}$ from a non-empty compact convex subset $\X_{i}$ of a finite-dimensional normed space $\mathcal{V}_{i}=(\real^{D_{i}},\norm{\cdot}_{i})$. This process results in joint action/strategy-profile $\bx=(\bx^{(1)},\ldots,\bx^{(N)})\in\X:=\prod_{i=1}^{N}\X_{i}\in\real^{D}$, where $D:=\sum_{i=1}^{N}D_{i}$. In order to highlight the action of player $i$, we write $\bx=(\bx^{(i)},\bx^{(-i)})$, where $\bx^{(-i)}=(\bx^{(j)})_{j\neq i}\in \mathcal{X}_{-i}:=\prod_{j\neq i}\X_{j}$. Working with the whole population, it is advantageous to consider the Euclidean normed space $\V:=(\prod_{i=1}^{N}\real^{D_{i}},\norm{\cdot})$, where $\norm{\bx}^{2}:=\sum_{i}\norm{\bx^{(i)}}^{2}_{i}$. Now, suppose that the population action at time $t$ is $\bx_{t}\in\mathcal{X}$. The payoff/reward agent $i$ received after $\bx_{t}$ is given by $\ru_{i}(\bx_{t}^{(i)},\bx_{t}^{(-i)})$, where $\ru_{i}:\X\rightarrow\real$ is the utility function of the  $i$. Throughout this paper, we assume the following standard regularity condition for the utility functions:
\begin{assum} 
	\label{Ass:aiishhfggfjhhdjjdhhdjjddd}
	For all $i\in [N]$ and $\bx^{(-i)}\in \X_{-i}$,
	$\ru_{i}((\cdot),\bx^{(-i)})$ is 
	concave  and $\bv:=(\bv^{(1)},\ldots,\bv^{(N)})$ is continuous where: $$\bv^{(i)}(\bx):=\nabla_{\bx^{(i)}}\ru_{i}(\bx),\quad i\in[N].$$
\end{assum}
\paragraph*{Coupled Resource Constraints} 
For a certain number $R>0$ of resources, we model the relation between agents' action and resource utilization by a function $\bg:\real^{D}\rightarrow\real^{R}$.
Throughout this paper, we assume that the function $\bg$ is subject to the following conditions:
\begin{assum} 
	\label{Ass:aiishhfggfjhhdjjdhhdjjddd2}
	For all $r\in [R]$, $\bg_{r}:\mathcal{X}\rightarrow\real$ is convex and differentiable, and the Jacobian matrix $\nabla \bg$ of $\bg$ is continuous.
\end{assum}
One may interpret the term $\bg_{r}(\bx)$ as the overload/congestion state of the resource $r\in [R]$ caused by the population action $\bx$. Since from an operational and sustainability point of view, overload has to be kept low and even avoided, it is desired that the population strategy is contained in $\mathcal{Q}:=\mathcal{C}\cap\mathcal{X}$, where $\mathcal{C}:=\{\bg(\bx)\leq 0\}$ denotes the resource constraints. For the feasibility of this goal, we assume that $\mathcal{C}$ is non-empty.
The following regularity condition on $\mathcal{Q}$ is useful for later purpose:
\begin{assum}[Slater's condition]
	\label{Eq:aioaoosjjshhdjjddss}
	There exists a point $\bx_{*}$ in $\textnormal{relint}{(\X)}$  
	such that $\bg(\bx_{*})< 0$, where $\text{relint}(\X)$ denotes the relative interior of $\X$.
\end{assum}
Notice that the compliance of the constraint $\bx\in\mathcal{C}$ depends not only on the strategy of a single agent, but also on the whole population. For this reason, one refers to this constraint as \textit{coupled constraint}. Correspondingly, NG subject to coupled inequality constraints $\mathcal{C}$, is also called NG with coupled constraints (NGCC). The set of the feasible strategy of the player $i$ given a joint action $\bx^{(-i)}$ of other agents is denoted by $\mathcal{Q}^{(i)}(\bx^{(-i)}):=\{\bx^{(i)}\in\mathcal{X}_{i}:~\bg(\bx)\leq 0\}$. 
 
\paragraph*{Generalized Nash Equilibrium (GNE)}
One of the central concepts in game theory is the so-called \textit{Nash equilibrium}. Specifically, NE denotes a feasible strategy profile for which no agent can improve her reward by unilaterally deviating from her strategy. For general games underlying coupled constraints, we formally define this notion as follows (see, e.g., \cite{Facchinei2007}):  
\begin{definition}[Generalized Nash Equilibrium (GNE)]
Given a NGCC $\Gamma$ and a resource constraint $\mathcal{C}$, $\bx_{*}\in\mathcal{Q}$ is said to be a (generalized) Nash equilibrium of $\Gamma$ with $\mathcal{C}$ if for each $i\in[N]$:
\begin{equation*}
\ru_{i}(\bx^{(i)}_{*},\bx^{(-i)}_{*})\geq \ru_{i}(\bx^{(i)},\bx^{(-i)}_{*}),\quad \forall \bx^{(i)}\in\mathcal{Q}(\bx^{(-i)}_{*}).
\end{equation*} 
The set of all generalized Nash equilibrium of $\Gamma$ is denoted by $\GNE(\Gamma)$
\end{definition}

\subsection{Mirror Map and Fenchel Coupling}
\label{Subsec:aiaissjshhsjshsss}
Throughout this work, we assume that the agents use first-order information of their utility function to make a decision. However, the first-order information is not necessarily a feasible strategy. For this reason, we specify in the following the class of maps helping an agent realizing her decision from the aforementioned information:
\begin{definition}[Regularizer/penalty function and Mirror Map]
Let $\Z$ be a compact convex subset of a Euclidean normed space $\V$, and $K>0$. We say $\uppsi:\Z\rightarrow\mathbb{R}$ is a $K$-strongly convex \textit{regularizer} (or \textit{penalty function}) on $\Z$, if $\uppsi$ is continuous and $K$-strongly convex on $\Z$, i.e., for all $\bx,\by\in \mathcal{Z}$ and $\lambda\in [0,1]$:
\begin{equation*}
\uppsi(\lambda \bx+(1-\lambda)\by)\leq \lambda \uppsi(\bx)+(1-\lambda)\uppsi(\by)-\tfrac{K}{2}\lambda (1-\lambda)\norm{\bx-\by}^{2}.
\end{equation*}
The mirror map $\bfPhi:\V^{*}\rightarrow\Z$ induced by $\uppsi$ is defined by:
$$\bfPhi(\by):=\argmax_{\bx\in\Z}\left\{\left\langle \by,\bx\right\rangle-\uppsi(\bx)\right\}.$$
\end{definition}
The following proposition, which is folklore in convex analysis, gives some basic properties of the mirror map:
\begin{proposition}
\label{Prop:aiaishshjfggfhdhddd}
Let $\uppsi$ be a $K$-strongly convex regularizer on a compact convex subset $\mathcal{Z}$ of a Euclidean normed space $\mathcal{V}$ inducing the mirror map $\bfPhi:\V^{*}\rightarrow\Z$, and let $\uppsi^{*}:\V^{*}\rightarrow\real$, $\by\mapsto\max_{\bx\in\Z}\left\{\left\langle \bx,\by\right\rangle-\uppsi(\bx)\right\}$ be the convex conjugate of $\uppsi$. Then:
\begin{enumerate}
\item $\bx=\bfPhi(\by)$ if and only if $\by\in\partial \uppsi(\bx)$. In particular $\text{im}(\bfPhi)=\text{dom}(\partial \uppsi)\supseteq \text{relint}(\Z)$.
\item $\uppsi^{*}$ is differentiable on $\V^{*}$ and $\nabla \uppsi^{*}(\by)=\bfPhi(\by)$.
\item $\bfPhi$ is $(1/K)$-Lipschitz continuous,
\end{enumerate}
where $\partial \uppsi(\bx)$ denotes the subgradient of $\uppsi$ at $\bx$.
\end{proposition}
For a proof of these facts, see e.g., Theorem 23.5 in \cite{Rockafellar1970} and Theorem 12.60(b) in \cite{Rockafellar1998}. 

To illustrate the concept of the mirror map, we provide in the following some examples of this map. First of all, the usual Euclidean projection is clearly an instance of the mirror map. 
Another example of the mirror map is the so called \textit{logit choice} $\bfPhi(\by)=\exp(\by)/\sum_{l=1}^{D}\exp(\by_{l})$, generated by the $1$-strongly convex regularizer $\uppsi(\bx)=\sum_{k=1}^{D}\bx_{k}\log \bx_{k}$ on the probability simplex $\Delta\subset(\real^{D},\norm{\cdot}_{1})$. This kind of mirror map is popular in the field of the decision making. 
Some non-standard examples of mirror maps are $\bfPhi(\by)=(\exp(\by_{k})/(1+\exp(\by_{k})))_{k\in [n]}$ induced by the Fermi-Dirac entropy $\uppsi(\bx)=\sum_{k=1}^{D}(\bx_{k}\log \bx_{k}+(1-\bx_{k})\log (1-\bx_{k})$ on $\Z=[0,1]^n$, and $\bfPhi(\bfY) = \exp(\bfY)/(1 + \norm{\exp(\bfY)}_{1})$ induced by the von-Neumann entropy $\uppsi(\bfX) = \text{tr}(\bfX \log \bfX) + (1 - \text{tr} \bfX) \log(1 - \text{tr}\bfX)$ on the set of positive semidefinite matrices equipped with the nuclear norm.

From the above examples, we can see that the mirror map provides a rich model for the decision-making process by generalizing the projection operator. Besides, another significant advantage of using the concept of the mirror map is a weaker dependence of dimension for the performance of the decision-making process by appropriately adapting this map to the underlying problem geometry. For instance, using the logit choice instead of Euclidean projection for realizing iterative first-order descent method for convex optimization problem on a simplex yields a convergence guarantee which depends logarithmically on the dimension, instead of on the square root of the underlying dimension (see e.g., \cite{Nemirovski2008}).      

As noticed in \cite{Mertikopoulos2016}, the following notion of "distance" is canonical to mirror map:
\begin{definition}[Fenchel Coupling]
	Let $\uppsi:\Z\rightarrow\real$ be a penalty function on a compact convex subset $\Z$ of a Euclidean normed space $\V$. The Fenchel coupling induced by $\uppsi$ is defined as
	$\rmF:\Z\times \V^{*}\rightarrow\real$, $(\boldp,\by)\mapsto\uppsi(\boldp)+\uppsi^{*}(\by)-\inn{\by}{\boldp}$.
\end{definition}
Some useful properties of the Fenchel coupling are stated in the following (for proof see \cite{Mertikopoulos2016}):
\begin{proposition}
	\label{Prop:aaisshhfjffjfjfff}
	Let $\rmF$ be the Fenchel coupling induced by a $K$-strongly convex regularizer on a compact convex subset $\mathcal{Z}$ of a Euclidean normed space $\V$. For $\boldp\in\mathcal{
		Z}$, $\by,\by^{'}\in \V^{*}$, we have:
	\begin{enumerate}
		\item $\rmF(\boldp,\by)\geq (K/2)\norm{\bfPhi(\by)-\boldp}^{2}$
		\item $\rmF(\boldp,\by^{'})\leq \rmF(\boldp,\by)+\inn{\bfPhi(\by)-\boldp}{\by^{'}-\by}+(1/2K)\norm{\by^{'}-\by}^{2}_{*}$
	\end{enumerate}
\end{proposition}

Throughout this work, we assume that each agent $i\in [N]$ possesses a $K_{i}$-strongly convex regularizer $\uppsi_{i}$ that induces the mirror map $\bfPhi_{i}$ and the Fenchel coupling $\rmF_{i}$. For the sake of simplicity, we assume $K_{i}=K$, for all $i\in [N]$. In order to emphasize the action of the whole population, we use the operator $\bfPhi:\V\rightarrow\X$, $\by\mapsto (\bfPhi_{1}(\by^{(1)}),\ldots,\bfPhi_{N}(\by^{(N)}))$ and the total Fenchel coupling $\rmF:\X\times \V^{*}\rightarrow\real_{\geq 0}$, $(\bx,\by)\rightarrow\sum_{i}\rmF_{i}(\bx_{i},y_{i})$.

\begin{algorithm}[htbp]
	\caption{Mirror Ascent Augmented Resource Pricing (MAARP)}
	\begin{algorithmic}
		\REQUIRE Step size sequence $(\gamma_{n})$, augmentation functions $(\uptheta_{n})$ 
		\REQUIRE Initial dual action $\boldY_{0}^{(i)}\in\V^{*}_{i}$, dual variable $\boldLambda_{0}\in\real_{\geq 0}^{R}$
		\FOR{$n=0,1,2,\ldots$}
		\STATE Population play $\boldX_{n}=\bfPhi(\boldY_{n})$
		\FOR{every player $i\in [N]$ }
		\STATE Observe the gradient utility feedback $\hat{\bv}^{(i)}_{n}$ given by \eqref{Eq:aakshsgsgsgsggssss}
		\STATE Query the gradient load $\nabla_{\boldX^{(i)}_{n}} \bg_{r}(\boldX_{n})$, $r\in [R]$
		\STATE Update the score vector by the rule \eqref{Eq:ajkajajshshsgsgsgsgsgsshhssgsgss} 
		\ENDFOR
		\FOR{every resource $r\in [R]$}
		\STATE Check its load $\bg_{r}(\boldX_{n})$
		\STATE Update the price by the rule \eqref{Eq:aassjhshshshsgdgdfdfdfdgdgdfdfd}
		\STATE Broadcast $\boldLambda_{n+1}$ to all players.
		\ENDFOR 
		\ENDFOR
	\end{algorithmic}
	\label{Alg:aoaishhjddhhddddeee}
\end{algorithm} 
An implication of Proposition \ref{Prop:aaisshhfjffjfjfff} important for later investigation is that the "convergence" of the sequence $(\bfPhi(\boldY_{n}))_{n}$ induced by another sequence $(\boldY_{n})_{n}$ w.r.t. $\rmF$, i.e., $\rmF(\boldp,\boldY_{n})\rightarrow 0$ for an $\boldp\in\mathcal{X}$, implies the convergence of $(\bfPhi(\boldY_{n}))_{n}$ w.r.t. the underlying norm. For later purpose, it is helpful to assume that the converse statement holds:
\begin{assum}[Reciprocity Condition]
	\label{Ass:aoasjsdhhggehegeeeerrrttt}
	For any $\boldp\in\X$ and any sequence $(\boldY_{n})_{n}$ in $\V^{*}$, it holds: $\bfPhi(\boldY_{n})\rightarrow \boldp~\Rightarrow~ \rmF(\boldp,\boldY_{n})\rightarrow 0$.
\end{assum}
This condition is standard in the literature of mirror descent (see \cite{Chen1993}).
\subsection{Online Mirror Descent}
\label{Subsec:aaksjsjshsshsgggshssss}
The foundation of the mechanism proposed in this work is given by the following iterate of the agent $i\in [N]$ defined as follows: 
\begin{equation}
\label{Eq:aoaaosjsdhdjdhhddss}
\boldX^{(i)}_{n+1}=\bfPhi_{i}(\boldY^{(i)}_{n+1}),~\boldY^{(i)}_{n+1}=\boldY^{(i)}_{n}+\gamma_{n}\bv^{(i)}(\boldX_{n}).
\end{equation}
This algorithm is a canonical extension of the standard mirror ascent algorithm within the framework of online learning \cite{Shalev-Shwartz2012}, in which the learner subsequently tries to optimize his apriori unknown time-variant regret/payoff-function using some problem-specific feedback (in our case: the first-order information of her utility function). 

For practical reasons, we assume that each agent $i$ does not know the exact gradient $\bv^{(i)}$. 
We can model this aspect, by modifying \eqref{Eq:aoaaosjsdhdjdhhddss} as follows:
\begin{equation}
\label{Eq:aoaaosjsdhdjdhhddss2}
\boldX^{(i)}_{n+1}=\bfPhi_{i}(\boldY_{n+1}^{(i)}),~\boldY^{(i)}_{n+1}=\boldY^{(i)}_{n}+\gamma_{n}\hat{\bv}^{(i)}_{n},
\end{equation}   
where:
\begin{equation}
\label{Eq:aakshsgsgsgsggssss}
\hat{\bv}^{(i)}_{n}=\bv^{(i)}(\boldX_{n})+\bM^{(i)}_{n+1},
\end{equation}
and $(\bM^{(i)}_{n})_{n\in\nat}$ be a $\mathcal{V}_{i}^{*}$-valued \textit{$\mathbb{F}$-martingale difference sequence}, i.e. for all $n\in\nat$:
\begin{enumerate}
	\item $\mathbb{F}$-Adaptedness: $\bM^{(i)}_{n}$ is $\mathcal{F}_{n}$-measureable for all $n\in\nat$,
	\item square-integrability: $\Erw[\norm{\bM^{(i)}_{n}}^{2}_{i,*}]<\infty$,
	\item conditionally zero mean: $\Erw[\bM^{(i)}_{n}|\mathcal{F}_{n-1}]=0$ a.s. .
\end{enumerate}
This noise model is quite general, since it does not only cover the i.i.d. zero mean noise with finite variance, but also noises with memory. An example is $\bM^{(i)}_{n}=\epsilon^{(i)}_{n}\epsilon^{(i)}_{n-1}$ where $(\epsilon^{(i)}_{n})_{n}$ i.i.d. mean zero RV and $\mathbb{F}$ is the corresponding filtration (containing the filtration) generated by the history of $(\epsilon^{(i)}_{n})$. 
\section{Mirror Ascent with Augmented Lagrangian}
\label{Sec:aiishhhsggdgdgdhhsjjshhssss}
In order to give the agents incentives for sustainable use of resources, our advice is to obligate the participants to pay additional cost for the amount of resources' utilization. In this regard, the cost can be given either in monetary unit, or in non-monetary unit such as system/network credit. A precondition for carrying out this obligation is, e.g., an agreement which each participant has to give upon entrance in the system. 

To be more specific, let us consider a time slot $n$. At $n$, the agent $i$ is obligate to pay in total $\sum_{r=1}^{R}\boldLambda_{n}^{r}\bg_{r}(\boldX^{(i)}_{n+1},\boldX^{(-i)}_{n})$ for a possible future action $\boldX^{(i)}_{n+1}\in\mathcal{X}_{i}$. In this expression, the variable $\boldLambda^{r}_{n}$ denotes the price of the resource $r$ at time $n$. This information is set and broadcasted to all agents by the resource itself. The specific setting rule of the price will be given in the next paragraph. So at time $n$, the utility function of agent $i$ becomes $\ru_{i}((\cdot),\boldX_{n})+\sum_{r=1}^{R}\boldLambda_{n}^{r}\bg_{r}((\cdot),\boldX^{(-i)}_{n+1})$,
and consequently, the gradient update \eqref{Eq:aoaaosjsdhdjdhhddss2} turns to:
\begin{equation}
\label{Eq:ajkajajshshsgsgsgsgsgsshhssgsgss}
\boldY^{(i)}_{n+1}= \boldY_{n}^{(i)}+\gamma_{n}\left(\hat{\bv}^{(i)}_{n}-\sum_{r=1}^{R}\boldLambda_{n}^{r}\nabla_{\boldX^{(i)}_{n}} \bg_{r}(\boldX^{(i)}_{n},\boldX^{(-i)}_{n})\right).
\end{equation}
Above formulation may indicate that the agents need to know the joint action for the instantaneous derivative of the resource $r$, i.e., $\nabla_{\boldX^{(i)}_{n}} \bg_{r}(\boldX^{(i)}_{n},\boldX^{(-i)}_{n})$. However, this is in general not true, since there is a large class of coupled constraints for which $\nabla_{\boldX^{(i)}_{n}} \bg_{r}(\boldX^{(i)}_{n},\boldX^{(-i)}_{n})$ only depends on the action of the agent $i$, i.e., $\boldX^{(i)}_{n}$. More detailed discussion on this aspect will later be given in Remark \ref{Rem:aaiaishshshdjjdhdhdggsss}.

Our proposal for the pricing of the resource $r$ at time $n+1$ is inspired by the augmented Lagrangian method (see e.g. \cite{Mahdavi1}) and is given as follows:
\begin{equation}
\label{Eq:aassjhshshshsgdgdfdfdfdgdgdfdfd}
\boldLambda_{n+1}^{r}\leftarrow\bfPi_{\real_{\geq 0}}\left(\boldLambda^{r}_{n}+\gamma_{n}\left[\bg_{r}(\boldX_{n})-[\nabla_{\boldLambda_{n}}\uptheta_{n}(\boldLambda_{n})]_{r}\right]\right),
\end{equation}
where $\uptheta_{n}:\real^{R}\rightarrow\real$ is a differentiable function called the augmentation function. In this work, we mainly consider the augmentation functions of the form:
\begin{equation}
\label{Eq:AugFunc}
\uptheta_{n}(\boldlambda):=\alpha_{n}\norm{\boldlambda}^{2}_{2}/2,\quad\text{where $\alpha_{n}>0$}.
\end{equation}

Notice that the update \eqref{Eq:aassjhshshshsgdgdfdfdfdgdgdfdfd} is decentralized since it can be done by the resource $r$ itself. Moreover, it is easy to implement since it consists of cheap operations. Also worth to notice is that the information needed for \eqref{Eq:aassjhshshshsgdgdfdfdfdgdgdfdfd} is the in-situ information about the previous price $\boldLambda^{r}_{n}$, and the congestion state $\bg_{r}(\boldX_{n})$. Clearly, the resource $r$ does not need for the latter knowledge about the joint action of the population. The low-complexity of \eqref{Eq:aassjhshshshsgdgdfdfdfdgdgdfdfd} and the type of the information sources discussed previously provide our pricing method flexibility to handle possible exit of - and the entrance of participants in the system.  

 Now, we give for better understanding the intuition behind \eqref{Eq:aassjhshshshsgdgdfdfdfdgdgdfdfd}. Without the augmentation term, \eqref{Eq:aassjhshshshsgdgdfdfdfdgdgdfdfd} is the Lagrangian dual ascent corresponding to the coupled constraints. The only unclear term is thus $\nabla_{\boldLambda_{n}}\uptheta_{n}(\boldLambda_{n})$. In the following, we discuss this term.  
 The role of the augmentation function's gradient is to prevent the fast increase of the price due to the accumulation of the previous dynamic. To understand why the latter occurrence is undesired, suppose first that the price is high so that the obligated cost term dominates over the utility term ($\hat{\bv}^{(i)}_{n}$) in \eqref{Eq:ajkajajshshsgsgsgsgsgsshhssgsgss}. Then, the agents will rather decide for resources with cheapest cost than for resources which give them the most benefit. Since the price of a resource in absence of the augmentation function is proportional to its congestion state, all agents might
at worst (e.g., in the case $N=D$ and $\bg$ equal to the identity function, where an action corresponds directly to resource
utilization choice) consume a single resource with the lowest congestion and cause therefore the latter's price and load to rise dramatically. Subsequently in the next time slot, they will all mutually utilized another less congested and cheaper resource, causing its price and its congestion to rise dramatically. This procedure will repeat and cause agents' consumption choice bounces at worst from a single resource to another one, and meanwhile violation of resource capacity constraints.

The corresponding method, which we call mirror ascent augmented resource pricing (MAARP), is given in detail in Algorithm \ref{Alg:aoaishhjddhhddddeee}. 
The population's iterate given in \ref{Alg:aoaishhjddhhddddeee} can be written more compact as:
\begin{equation}
\label{Eq:aiashsggsgssffssf}
\begin{split}
\boldX_{n+1}&=\bfPhi(\boldY_{n+1}),~ \boldY_{n+1}=\boldY_{n}+\gamma_{n}\left(\hat{\bv}_{n}-[\nabla \bg(\boldX_{k})]^{\T}\boldLambda_{k} \right)\\
\boldLambda_{n+1}&=\bfPi_{\real^{R}_{\geq 0}}\left( \boldLambda_{n}+\gamma_{n}\left[ \bg(\boldX_{n})-\nabla_{\boldLambda_{n}}\uptheta_{n}(\boldLambda_{n})\right]  \right)  
\end{split}
\end{equation}

Now, we we want to argue that for large class of coupled constraints, our method is decentralized:
\begin{remark}[Decentralization Aspect]
\label{Rem:aaiaishshshdjjdhdhdggsss}
Suppose that the coupled constraint is quasi-affine, i.e. the function $\bg$ is given by $\bg(\bx)=\bfA\rmh(\bx)-\mathbold{b}$,
where $\bfA\in\real^{R\times \sum_{i=1}^{N}D_{i}}$, $\mathbold{b}\in\real^{R}$, and $\rmh(\bx)=[\rmh_{1}(\bx^{(1)})^{\T},\ldots,\rmh_{N}(\bx^{(N)})^{\T}]$, and where for any $i\in [N]$, $\rmh_{i}:\real^{D_{i}}\rightarrow\real^{D_{i}}$ is differentiable and available only to agent $i$. Moreover, denote $\bfA:=[\bfA_{1},\ldots,\bfA_{N}]$,
where for all $i\in [N]$, $\bfA_{i}\in\real^{R\times D_{i}}$ is the matrix which is available only to the agent $i$. In this setting, the iterate of agent $i$ is given specifically by $\boldX_{n+1}^{(i)}=\bfPhi_{i}(\boldY_{n}^{(i)}+\gamma_{n}[\hat{\bv}^{(i)}_{n}-\nabla\mathbf{h}_{i}(\boldX_{n}^{(i)})^{\T}\bfA_{i}^{\T}\boldLambda_{n}])$, where $\nabla\mathbf{h}_{i}(\bx^{(i)})$ denotes the Jacobian of $\mathbf{h}_{i}$ at $\bx^{(i)}$.
Therefore, the iterate of each agent consists of local information such as first-order feedback $\hat{\bv}^{(i)}_{n}$, the constraint matrix $\bfA_{i}$, and the Jacobian of $\mathbf{h}_{i}$ at the iterate of agent $i$, and the prices $(\boldLambda^{r}_{n})_{r\in [R]}$, set by the resources $r\in [R]$. Furthermore, the latter is updated parallelly by each resource. 
\end{remark}
\begin{remark}[Difference between MAARP and Alg. 2 in \cite{Paccagnan2017}]
	\label{Rem:aaiishssggdhddgdggdgdd}
Besides the fact that the feedback in MAARP is noisy, a difference between Algorithm 2 in \cite{Paccagnan2017} and MAARP is that Algorithm 2 in \cite{Paccagnan2017} uses the Euclidean projection, while MAARP uses a mirror map which is a generalization (see Subsection \eqref{Subsec:aiaissjshhsjshsss} and the introduction). 
As we will see later, the advantage of the use of a mirror map is a weak dependence of the algorithm performance on the dimension of the strategy space (see Remark \ref{Rem:aaiisgsgsgsffsdsdss} and Section \ref{Sec:NumExp}). Another difference between Algorithm 2 in \cite{Paccagnan2017} and the MAARP is that the price update in Algorithm 2 in \cite{Paccagnan2017} requires two consecutive congestion states of a resource, i.e. $\bg_{r}(\boldX_{n})$ and $\bg_{r}(\boldX_{n+1})$, while MAARP requires only $\bg_{r}(\boldX_{n})$. By this reason, the price update $\boldLambda^{r}_{n+1}$ of MAARP can, in contrast to Algorithm 2 in \cite{Paccagnan2017}, be done in parallel with the population update $\boldX_{n+1}$.  
\end{remark}

\section{Convergence Analysis of MAARP}
\label{Sec:aooaahhsjsshhddjjddd}
In this section, we investigate the convergence of the primal iterate of MAARP to a generalized Nash equilibrium (GNE) of the game of our interest. The corresponding convergence result implies in particular, that our pricing method can ensure that the strategy of the non-cooperative online learning agents' is sustainable w.r.t. the coupled (resource) constraints.

The GNE of our interest is the so-called variational Nash equilibrium (VNE) of the NGCC $\Gamma$ with the constraints $\mathcal{C}$. 
This subclass of GNE arises naturally with the first-order dynamic (w.r.t. $\bv$) structure of MAARP as it is defined as the solution $\SOL(\mathcal{Q},\bv)$ of the well-known variational inequality $\VI(\mathcal{Q},\bv)$ (see Definition \ref{Def:VI}). To improve the reading flow, we defer the detailed discussion of these aspects to the appendix.

In order to keep the argumentation short by avoiding the notion of convergence of a sequence to a set, we consider the case where the solution of $\VI(\mathcal{Q},\bv)$ (and respectively $\VI(\mathcal{X}\times \real_{\geq 0}^{R},\tilde{\bv})$) is unique. This property holds as asserted by Proposition \ref{Prop:aiaisshshjdhdggddd3}, if $v$ is strictly monotone, which holds if the utility function is strictly convex, in the sense that:
\begin{equation*}
\ru_{i}((\cdot),\bx^{(-i)})~\text{strictly convex}\quad\forall i\in[N],~\bx^{(-i)}\in\mathcal{X}_{-i}
\end{equation*} 
For the analysis in this section, we use the notation $\bM_{n}:=(\bM_{n}^{(i)})_{i}$ and $\sigma_{n}^{2}:=\Erw[\norm{\bM_{n}}_{*}^{2}]$.

We have the following convergence statement for the iterate of MAARP:
\begin{theorem}
	\label{Thm:aoasjsjsjdhdjshshs}
	Suppose that Assumption \ref{Ass:aoasjsdhhggehegeeeerrrttt} holds. Moreover, suppose that $\bv$ is, in addition, strictly monotone. Let the augmentation function $\uptheta_{n}$ be given by \eqref{Eq:AugFunc}.
	If the step size sequence $(\gamma_{n})_{n}$ satisfies the properties:
	\begin{align}
	&\sum_{k=0}^{\infty}\gamma_{k}=\infty,\quad\quad\sum_{k=0}^{\infty}\gamma_{k}^{2}<\infty\label{Eq:ClosedLoop}\\
	&
	\sum_{k=0}^{\infty}\gamma_{k}^2\sigma_{k+1}^{2}<\infty
	\label{Eq:oaaosjsjdhpllalaala}
	\end{align}
	and if $(\alpha_n)$ and $(\gamma_{n})$ satisfy:
	\begin{equation}
	\label{Eq:aoaoshshsjjhdggdhhss2}
	\sum_{k=0}^{\infty}\gamma_{k}\alpha_{k}<\infty
	\end{equation} 
	and for large enough $k\in\nat_{0}$:
	\begin{equation}
	\label{Eq:aooajsjsjskkjdjjjddd}
	\gamma_{k}\left(\alpha_{k}^{2}+\tfrac{C_{1}^{2}}{K}\right)-\tfrac{\alpha_{k}}{2}\leq 0,
	\end{equation}
	then the primal iterate $(\boldX_{n})$ of MAARP converges to the unique variational Nash equilibrium $x_{*}=\SOL(\mathcal{Q},\bv)$.
\end{theorem}
Before we provide the proof of the above statement, we give some examples of step size sequences that fulfill the assumptions of the above Theorem: 
\begin{remark}
	\label{Rem:Step}
	Assume that the noise is persistent, i.e., there exists $\sigma>0$ s.t. $\sigma_{k}\leq\sigma$ for all $k\in\nat$. The condition \eqref{Eq:oaaosjsjdhpllalaala} turns to $\sum_{k=0}^{\infty}\gamma_{k}^{2}<\infty$. Therefore, we can eliminate in this case the redundant condition \eqref{Eq:oaaosjsjdhpllalaala}. Furthermore, notice that \eqref{Eq:oaaosjsjdhpllalaala} includes the possibilities that $\gamma_{n}=\Theta(n^{-p})$, $p\in(1/2,1]$, but rules out the possibilities that $\gamma_{n}=\Theta(n^{-p})$, $p\in[0,1/2]$. Now, let be $\gamma_{n}=\gamma/(n+1)^{p}$, where $\gamma>0$ and $p\in(1/2,1]$ arbitrary. If we choose $\alpha_{n}=\alpha\gamma_{n}$, for an $\alpha>0$, then \eqref{Eq:aoaoshshsjjhdggdhhss2} is fulfilled. Moreover, we have:
	\begin{equation*}
	\gamma_{n}(\alpha_{n}^{2}+\tfrac{C_{1}^{2}}{K})-\tfrac{\alpha_{n}}{2}=\gamma_{n}\left[\alpha^{2}\gamma_{n}^{2}+\tfrac{C_{1}^{2}}{K}-\tfrac{\alpha}{2}\right].
	\end{equation*} 
	In case $\alpha>2C_{1}^{2}/K$, we can find $c>0$ s.t.:
	\begin{equation*}
	\gamma_{n}(\alpha_{n}^{2}+\tfrac{C_{1}^{2}}{K})-\tfrac{\alpha_{n}}{2}\leq\gamma_{n}\left[\alpha^{2}\gamma_{n}^{2}-c\right].
	\end{equation*}
	Since the R.H.S. of the above inequality is negative for large $n$, \eqref{Eq:aooajsjsjskkjdjjjddd} is fulfilled by the choice $\alpha>2C_{1}^{2}/K$.
\end{remark}
In the following, we relate the noise condition with the usual one given in the literature:  
\begin{remark}
\label{Rem:aaauiusgsghsgshshsddd}
By the Fubini's Theorem and the tower property one can see that the condition \eqref{Eq:oaaosjsjdhpllalaala} implies the condition that a.s. $\sum_{k=0}^{\infty}\gamma_{k}^{2}\Erw_{k}[\norm{\bM_{k+1}}^{2}_{*}]<\infty$, which is often used in the martingale analysis. 
\end{remark}

In order to differentiate our approach from one of the closest previous prior works, we give the following remarks:
\begin{remark}[Relation to Stochastic Approximation Theory]
At first sight, one may think that the dynamic \eqref{Eq:aiashsggsgssffssf} is an instance of the stochastic approximation algorithm (see e.g. \cite{Benaim1999,Bharath1999,Borkar2008}) having the archetypical form:
\begin{equation}
\label{Eq:aiaiassgfsfsdfsddssssaaa}
\boldZ_{n+1}=\boldZ_{n}+\gamma_{n}\left[\bfh(\boldZ_{n})+\bM_{n+1}\right],
\end{equation}  
for a Lipschitz continuous vector field $h$. So the question might arise whether one can immediately obtain Theorem \ref{Thm:aoasjsjsjdhdjshshs} using the stochastic approximation theory, whose approach ("ODE approach") consists of considering \eqref{Eq:aiaiassgfsfsdfsddssssaaa} as a Cauchy-Euler approximation of the ordinary differential equation (ODE) $\dot{\boldZ}_{t}=\bfh(\boldZ_{t})$. However, taking a detailed look at \eqref{Eq:aiashsggsgssffssf}, one can recognize that our proposed algorithm differs from \eqref{Eq:aiaiassgfsfsdfsddssssaaa} by the non-linear mappings, i.e., $\bfPhi$ and $\bfPi_{\real^{R}_{\geq 0}}$, ensuring that the corresponding dynamic remains in the feasible sets. Moreover, in contrast to stochastic approximation theory, we do not require any Lipschitz condition on the vector field $\bv$, which is needed for the uniqueness of the solution of the corresponding ODE. Even if we require the Lipschitz continuity $\bv$, taking a similar ODE approach as done in the stochastic approximation theory by defining the ODE approximation of \eqref{Eq:aiashsggsgssffssf} as:
\begin{equation*}
\begin{split}
\boldX_{t}&=\bfPhi(\boldY_{t}),\quad \dot{\boldY}_{t}=\bv(\boldX_{t})-[\nabla \bg(\boldX_{t})]^{\T}\boldLambda_{t}\\
\boldLambda_{t}&=\bfPi_{\real^{R}_{\geq 0}}\left(\boldGamma_{t} \right),\quad \dot{\boldGamma}_{t}=\bg(\boldX_{t})-\nabla_{\boldLambda}\uptheta_{t}(\boldLambda_{t}).
\end{split} 
\end{equation*}
would surely require intricate argumentation: Starting with showing that the solution of above ODE uniquely exists -- which is not immediately follows since a mirror map is in general not invertible. Nevertheless, the work \cite{Mertikopoulos2018} investigating the dynamic \ref{Eq:aoaaosjsdhdjdhhddss2} constituting the fundament of our proposed algorithm \eqref{Eq:aiashsggsgssffssf} follows the intricate ODE approach by using techniques provided in \cite{Benaim1999}. We give a more detailed comment on this aspect in Remark \ref{Rem:aiaishssgggsggsss}.

Besides, the requirements \eqref{Eq:ClosedLoop} and 	\eqref{Eq:oaaosjsjdhpllalaala} which are the usual summability condition in the stochastic approximation theory (see e.g. equation (2) in \cite{Bharath1999}) might lead some readers to think that Theorem \ref{Thm:aoasjsjsjdhdjshshs} is an easy consequence of the stochastic approximation theory. However, in order to derive the convergence result, we need to postulate additional requirements: The requirement \eqref{Eq:aoaoshshsjjhdggdhhss2} that ensures the decreasing influence of the non-equilibrium of the prices, and the requirement \eqref{Eq:aooajsjsjskkjdjjjddd} that ensures that the price update tracks the population dynamic.    
\end{remark}
\begin{remark}[Relation to \cite{Mertikopoulos2018}]
\label{Rem:aiaishssgggsggsss}
	One may think that by defining a canonical new game with a new player controlling the dual variable 
	Theorem \ref{Thm:aoasjsjsjdhdjshshs} is a simple assertion of Theorem 4.7 \cite{Mertikopoulos2018}. On the contrary, this claim is not valid, since Theorem 4.7 in \cite{Mertikopoulos2018} relies on the fact that the constraint set of each player is compact, while the constraint set $\real^{R}_{\geq 0}$ of the dual variable is unbounded. Also by the latter, we cannot imitate the approach done in \cite{Mertikopoulos2018} based on the theory provided in \cite{Benaim1999}. This hurdle motivates us to search for another way (Lemma \ref{Lem:aaosjjshhdhdhdhd}) to generate the convergence statement (Theorem \ref{Thm:aoasjsjsjdhdjshshs}) from the recurrence result (Lemma \ref{Lem:aoaosjsjshdhdjshssjjddd}). Our approach (Lemma \ref{Lem:aaosjjshhdhdhdhd}) is much simpler than that given in \cite{Mertikopoulos2018} and can also be used to generate Theorem 4.7 in \cite{Mertikopoulos2018}.
	Another difference of our work to \cite{Mertikopoulos2018} is our weaker noise assumption. In \cite{Mertikopoulos2018}, it is assumed that there exists $\sigma>0$ s.t. $\Erw_{n}[\norm{\bM_{n+1}}^{2}_{*}]\leq \sigma^{2}$ a.s. for all $n\in\nat_{0}$. Assuming \eqref{Eq:ClosedLoop}, which is also done in \cite{Mertikopoulos2018}, and applying Fubini's Theorem and the tower property, the latter observation implies \eqref{Eq:oaaosjsjdhpllalaala}. 
\end{remark} 
\subsection{Bound for Primal-Dual Iterate}
\label{Subsec:aoaoaosjjssjsjkkddd}
The first step to prove Theorem \ref{Thm:aoasjsjsjdhdjshshs} is to investigate the distance between the primal-dual iterate of the MAARP to the solution of $\VI(\mathcal{X}\times \real_{\geq 0}^{R},\tilde{\bv})$. As a distance function, we use:
\begin{equation*}
\tilde{\rmF}((\bx,\boldlambda),(\bfPhi(\by),\tilde{\boldlambda})):=\rmF(\bx,\by)+(\norm{\boldlambda-\tilde{\boldlambda}}^{2}_{2}/2)
\end{equation*}
where $\bx\in\mathcal{X}$, $\by\in \real^{D}$ and $\boldlambda\in\real^{R}$. The following result gives a bound for:
\begin{equation*}
\mathcal{V}_{n}(\bz):=\tilde{\rmF}(\bz,\boldZ_{n})-\tilde{\rmF}(\bz,\boldZ_{0}),
\end{equation*}
where $\bz:=(\bx,\boldlambda)\in\mathcal{X}\times \real^{R}_{\geq 0}$ and $\boldZ_{n}=(\boldX_{n},\boldLambda_{n})$:
\begin{theorem}
\label{Thm:UpperBoundFench}
Let $C_{1},C_{2},C_{3}>0$ be constants s.t. for all $\bx\in\mathcal{X}$ and $\boldlambda\in \real^{R}_{\geq 0}$ it holds:
\begin{equation}
\label{Eq:aoaoshshjhsjjss}
\norm{\nabla \bg(\bx)^{\T}\boldlambda}_{*}\leq C_{1}\norm{\boldlambda}_{2},\norm{\bv(\bx)}_{*}\leq C_{2}, \norm{\bg(\bx)}_{2}\leq C_{3},
\end{equation}  
and suppose that for all $n$, $\uptheta_{n}$ is continuously differentiable and $K_{n}$-strongly convex. For all $\bz:=(\bx,\boldlambda)\in\mathcal{X}\times \real^{R}_{\geq 0}$, it holds:
\begin{equation}
\begin{split}
\mathcal{V}_{n}(\bz)&\leq\sum_{k=0}^{n-1}\gamma_{k}\upeta_{k}(\bz)+2\left(\tfrac{C_{2}^{2}}{K}+\tfrac{C_{3}^{2}}{2}\right)\sum_{k=0}^{n-1}\gamma_{k}^2\\
&~~~+\sum_{k=0}^{n-1}\gamma_{k}\Uppsi_{k}(\boldLambda_{k},\boldlambda)+\boldS_{n}(\bx)+\tfrac{2}{K}\boldR_{n}
\end{split}
\label{Eq:aoaosjjsdjhdjjdhffffff}
\end{equation}
where:
\begin{align}
\upeta_{k}(\bz)&:=\inn{\boldZ_{k}-\bz}{\tilde{\bv}(\boldZ_{k})},\nonumber\\
\begin{split}
\label{Eq:SumNoise}
\boldS_{n}(\bx)&:=\sum_{k=0}^{n-1}\gamma_{k}\tilde{\bM}_{k}(\bx),~\tilde{\bM}_{k}(\bx):=\inn{\boldX_{k}-\bx}{\bM_{k+1}},\\ \boldR_{n}&:=\sum_{k=0}^{n-1}\gamma_{k}^{2}\norm{\bM_{k+1}}_{*}^{2},
\end{split}
\\
\Uppsi_{k}(\boldLambda_{k},\boldlambda)&=\uptheta_{k}(\boldlambda)-\tfrac{K_{k}}{2}\norm{\boldLambda_{k}-\boldlambda}_{2}^{2}\nonumber
\\
&+\left[\gamma_{k}\left(\norm{\nabla_{\boldLambda_{k}}\uptheta_{k}(\boldLambda_{k})}^{2}_{2}+\tfrac{C_{1}^{2}}{K}\norm{\boldLambda_{k}}^{2}_{2}\right)-\uptheta_{k}(\boldLambda_{k})\right]\nonumber
\end{align}
\end{theorem}
\begin{proof}
By inserting the iterate of the MAARP into the bound given in Proposition \ref{Prop:aaisshhfjffjfjfff}, by using triangle inequality, by the inequality $(\sum_{i=1}^{K}a_{i})^{2}\leq K\sum_{i=1}^{K}a_{i}^{2}$, and by the assumption \eqref{Eq:aoaoshshjhsjjss}, it is straightforward to obtain the following bound for $\mathcal{E}^{(1)}_{k}(\bx):=\rmF(\bx,\boldY_{k})$:
\begin{align*}
\mathcal{E}^{(1)}_{k+1}(\bx)&-\mathcal{E}^{(1)}_{k}(\bx)
\leq\gamma_{k}\inn{\boldX_{k}-\bx}{\bv(\boldX_{k})-[\nabla \bg(\boldX_{k})]^{T}\boldLambda_{k}}\\&+\gamma_{k}\tilde{\bM}_{k+1}+\tfrac{\gamma_{k}^{2}}{K}\left(C_{1}^{2}\norm{\boldLambda_{k}}_{2}^{2}+2(C_{2}^{2}+\norm{\bM_{k+1}}_{*}^{2})\right)
\end{align*}
Now, we sum the above inequality over all $k=0,\ldots,n-1$ and subsequently telescope, and obtain the following upper bound for $\mathcal{V}_{n}^{(1)}(\bx):=\rmF(\bx,\boldY_{n})-\rmF(\bx,\boldY_{0})$:
\begin{align*}
\mathcal{V}_{n}^{(1)}(\bx)&\leq\sum_{k=0}^{n-1}\gamma_{k}\left[\inn{\boldX_{k}-\bx}{\bv(\boldX_{k})-[\nabla \bg(\boldX_{k})]^{T}\boldLambda_{k}}\right]\\
&+\sum_{k=0}^{n-1}\tfrac{\gamma_{k}^2C_{1}^{2}}{K}\norm{\boldLambda_{k}}^{2}_{2}+\boldS_{n}(\bx)+\boldR_{n}+\tfrac{2 C_{2}^{2}}{K}\sum_{k=0}^{n-1}\gamma_{k}^2.
\end{align*}

Next, similar computation gives the following bound for $\mathcal{V}_{n}^{(2)}(\boldlambda):=(\norm{\boldLambda_{n}-\boldlambda}_{2}^{2}-\norm{\boldLambda_{0}-\boldlambda}_{2}^{2})/2$:
\begin{equation}
\label{Eq:ajajajghggdffsgsgsgsgsgs}
\begin{split}
\mathcal{V}_{n}^{(2)}(\boldlambda)&\leq \sum_{k=0}^{n-1}\gamma_{k}\inn{\boldLambda_{k}-\boldlambda}{\bg(\boldX_{k})}\\
&-\sum_{k=0}^{n-1}\gamma_{k}\inn{\boldLambda_{k}-\boldlambda}{\nabla_{\boldLambda_{k}}\uptheta_{k}(\boldLambda_{k})}\\
&+\sum_{k=0}^{n-1}\gamma_{k}^{2}(C_{3}^{2}+\norm{\nabla_{\boldLambda_{k}}\uptheta_{k}(\boldLambda_{k})}_{2}^{2}).\nonumber
\end{split}
\end{equation}
Now, since $\uptheta_{k}$ is $K_{k}$-strongly convex for all $k$, we have:
\begin{equation*}
\inn{\boldlambda-\boldLambda_{k}}{\nabla_{\boldLambda_{k}}\uptheta_{k}(\boldLambda_{k})}\leq\uptheta_{k}(\boldlambda)-\uptheta_{k}(\boldLambda_{k})-\tfrac{K_{k}}{2}\norm{\boldLambda_{k}-\boldlambda}^{2}_{2}.
\end{equation*}
By setting this inequality into \eqref{Eq:ajajajghggdffsgsgsgsgsgs}, we have:
\begin{align*}
\mathcal{V}_{n}^{(2)}&(\boldlambda)\leq \sum_{k=0}^{n-1}\gamma_{k}\inn{\boldLambda_{k}-\boldlambda}{\bg(\boldX_{k})}+\sum_{k=0}^{n-1}\gamma_{k}[\uptheta_{k}(\boldlambda)-\uptheta_{k}(\boldLambda_{k})]\\
&-\sum_{k=0}^{n-1}\tfrac{\gamma_{k}K_{k}}{2}\norm{\boldLambda_{k}-\boldlambda}^{2}_{2}+\sum_{k=0}^{n-1}\gamma_{k}^{2}(C_{3}^{2}+\norm{\nabla_{\boldLambda_{k}}\uptheta_{k}(\boldLambda_{k})}_{2}^{2}).\label{Eq:oajjshhjjhdhjdhdjff}
\end{align*}
Combining the bounds for $\mathcal{V}_{n}^{(1)}(\bx)$ and $\mathcal{V}_{n}^{(2)}(\boldlambda)$, we obtain the desired bound for $\mathcal{V}_{n}(\bx,\boldlambda)=\mathcal{V}_{n}^{(1)}(\bx)+\mathcal{V}_{n}^{(2)}(\boldlambda)$.
\end{proof}
Regarding the estimate given in Theorem \ref{Thm:UpperBoundFench}, one possibility to gain control over $\mathcal{V}_{n}$ is to eliminate the dependency of the upper bound \eqref{Eq:aoaosjjsdjhdjjdhffffff} on $\boldLambda_{k}$, $k\in [n-1]_{0}$. We accomplish this goal by choosing the augmentation functions as in \eqref{Eq:AugFunc}, 
each for all $n\in\nat_{0}$, a $\alpha_{n}$-strongly convex augmentation function, yielding:
\begin{equation}
\begin{split}
2\Uppsi_{n}(\boldLambda_{n},\boldlambda)&=\alpha_{n}\norm{\boldlambda}^{2}_{2}+\beta_{n}\norm{\boldLambda_{n}}_{2}^{2}-\alpha_k\norm{\boldLambda_{n}-\boldlambda}_{2}^{2}\\ \tfrac{\beta_{n}}{2}&:=\left[\gamma_{n}\left(\alpha_{n}^{2}+\tfrac{C_{1}^{2}}{K}\right)-\tfrac{\alpha_{n}}{2}\right],
\end{split}
\label{Eq:AugFuncCond}
\end{equation}
and subsequently ensuring $\beta_{k}\leq 0$, by appropriately choosing $\alpha_{k}$ and $\gamma_{k}$ (see \eqref{Eq:aooajsjsjskkjdjjjddd}).
\subsection{Control over Noise}
In order to gain control over the disturbance due to the noise in the first-order feedback represented by the sums $\boldS_{n}(\bx)$ and $\boldR_{n}$ occurring in the upper bound \eqref{Eq:aoaosjjsdjhdjjdhffffff}, we apply the Doob's Martingale Convergence Theorem known in the literature of the martingale theory (see e.g., \cite{Hall1980}):
\begin{lemma}
\label{Lem:MartNoise}
Suppose that $(\bM_{n})$ is a square integrable $\real^{D}$-valued $\mathbb{F}$-martingale difference sequence  and that \eqref{Eq:oaaosjsjdhpllalaala} holds.
Then for the partial sums $\boldS_{n}=\boldS_{n}(\bx)$ and $\boldR_{n}=\boldR_{n}(\bx)$, $n\geq 0$, given in \eqref{Eq:SumNoise}, it holds:
\begin{enumerate}
\item $(\boldS_{n})_{n}$ is a mean zero square integrable $\mathbb{F}$-martingale and $(\boldR_{n})_{n}$ is a non-negative $\mathbb{F}$-sub-martingale
\item There exists a square integrable real RV $R_{\infty}$ and an integrable real RV s.t. $(\boldS_{n})_{n}$ converges a.s. and in $L^{2}$ to $S_{\infty}$ and $(\boldR_{n})_{n}$ converges a.s. and in $L^{1}$ to $R_{\infty}$. 
\end{enumerate}
\end{lemma}
\begin{proof} 
		It holds since $\bM_k$ is square integrable for all $k$:
		\begin{align*}
		\Erw[\abs{\boldR_{n}}]=\sum_{k=0}^{n-1}\gamma^{2}_{k}\Erw[ \norm{\bM_{k+1}}_{*}^2]<\infty.
		\end{align*}
		Moreover by the triangle inequality and the H\"older's inequality, it yields:
		\begin{align*}
		\abs{\boldS_{n}}^2&\leq n\sum_{k=0}^{n-1}\gamma_{t}^2\abs{\tilde{\bM}_{k}}^2
		\leq n\sum_{k=0}^{n-1}\gamma_{k}^2 \norm{X_{k}-\bx}^2
		\norm{\bM_{k+1}}^{2}_{*}\\
		&\leq T C_{\X}^2 \boldR_{n},	 
		\end{align*}
		where $C_{\mathcal{X}}>0$ is a constant, whose existence is ensured by the compactness of $\mathcal{X}$.
		This observation and the integrability of $\boldR_{n}$ give the square integrability of $\boldS_{n}$.
		
		Now, since $\norm{\bM_{k+1}}^{2}$ is non-negative, it follows immediately that $\Erw_{n-1}\left[\boldR_{n}\right]\geq R_{n-1}$. This observation and the apparent fact that $(\boldR_{n})$ is $\mathbb{F}$-adapted give that $(\boldR_{n})$ is a sub-martingale. 
		
		Next, we show that $(\boldS_{n})$ is a martingale.  Since for all $k\in\nat_{0}$, $X_{k}$ is a measureable function of $(\bM_{\tau})_{\tau\leq k}$, it follows that $(\boldX_{n})$ is adapted to $\mathbb{F}$. Consequently, $(\boldS_{n})_{n}$ is adapted to $\mathbb{F}$. Moreover, the fact that $\boldX_{n}$ is $\mathcal{F}_{n}$-measureable asserts:
		\begin{equation}
		\label{Eq:sskallaa}
		\begin{split}
		\Erw_{k}[\tilde{\bM}_{k}]
		=\inn{\boldX_{k}-\bx}{\Erw_{k}\left[\bM_{k+1}\right]}=0.
		\end{split}
		\end{equation} 
		So, consequently we have as desired, $\Erw_{n}\left[S_{n+1}\right]=\boldS_{n}$. At last, it is immediate to see that this fact and \eqref{Eq:sskallaa} assert that $\Erw[\boldS_{n}]=0$.

Now we prove the statement 2). For $k,\tilde{k}\in\nat_{0}$ with $k< \tilde{k}$, we have by the tower property, by the fact that $\tilde{\bM}_{k}$ is $\mathcal{F}_{\tilde{k}}$-measureable, and by \eqref{Eq:sskallaa}, $\Erw[\tilde{\bM}_{k}\tilde{\bM}_{\tilde{k}}]=\Erw[\Erw_{\tilde{k}-1}[\tilde{\bM}_{k}\tilde{\bM}_{\tilde{k}}]]=\Erw[\tilde{\bM}_{k}\Erw_{k}[ \tilde{\bM}_{\tilde{k}} ]]=0$.
		This asserts that:
		\begin{align*}
		\Erw[\abs{\boldS_{n}}^{2}]&=\sum_{k=0}^{n-1}\gamma_{k}^2\Erw[|\tilde{\bM}_{k}|^2]+\sum_{k\neq \tilde{k}}\gamma_{k}\gamma_{\tilde{k}}\Erw[\tilde{\bM}_{k}\tilde{\bM}_{\tilde{k}}]\\
		&\leq C_{\X}^{2}\sum_{k=0}^{n-1}\gamma_{k}^2\Erw[\norm{\bM_{k}}^{2}_{*}].
		\end{align*}
		To continue, notice that assumption \eqref{Eq:oaaosjsjdhpllalaala} asserts that $\sum_{k=0}^{\infty}\gamma_{k}^{2}\Erw[\norm{\bM_{k+1}}^{2}_{*}]<\infty$. So, this observation and above inequality assert that $(|\boldS_{n}|^2)$ converges uniformly. Therefore, by Martingale Convergence Theorem, we obtain the statement in 2) on  $(|\boldS_{n}|^2)$. Now, notice that the fact that $(\boldR_{n})$ is uniformly integrable follows immediately from \eqref{Eq:oaaosjsjdhpllalaala}. Consequently the remaining statement in 2) follows immediately from the Martingale Convergence Theorem. 
\end{proof}

\subsection{Convergence Proof}
\label{Subsec:aoaojskksjjskkdd}
The remaining auxiliary statement needed to prove Theorem \ref{Thm:aoasjsjsjdhdjshshs} is Proposition \ref{Prop:aiaisshshjdhdggddd} given in the Appendix. This proposition asserts that in order to show the convergence of the primal iterate to the variational Nash equilibrium of the NGCC $\Gamma$ with the constraints $\mathcal{C}$, we can instead investigate the convergence behavior of both the population - and the price iterate to the solution of the extended variational inequality $\VI(\mathcal{X}\times \real_{\geq 0}^{R},\tilde{\bv})$, where $\tilde{\bv}$ is the KKT operator corresponding to $\VI(\mathcal{Q},\bv)$ given in \eqref{Eq:iaaiisjsjdhhdhdhdhdjjsss}. A detailed discussion on this aspect is given in the Subsection \ref{Subsec:Decoup} in the Appendix.

One step in the direction of the proof of our main result (Theorem \ref{Thm:aoasjsjsjdhdjshshs}) is to combine the results obtained in the preceding subsections and show the recurrence of $\boldZ_{n}$ around the solution of the variational inequality $\VI(\mathcal{X}\times \real_{\geq 0}^{R},\tilde{\bv})$:
\begin{lemma}
	\label{Lem:aoaosjsjshdhdjshssjjddd}
	Suppose that the Assumptions \ref{Ass:aiishhfggfjhhdjjdhhdjjddd}, \ref{Ass:aiishhfggfjhhdjjdhhdjjddd2}, and \ref{Eq:aioaoosjjshhdjjddss} hold, and that $v$ is strictly monotone. Let the augmentation function $\uptheta_{t}$ be given by \eqref{Eq:AugFunc}.
	If $(\gamma_{n})_{n\in\nat_{0}}$, $(\bM_{n})_{n\in\nat}$, and $(\alpha_{n})_{n\in\nat_{0}}$ fulfill \eqref{Eq:oaaosjsjdhpllalaala}, \eqref{Eq:ClosedLoop}, \eqref{Eq:aoaoshshsjjhdggdhhss2}, and \eqref{Eq:aooajsjsjskkjdjjjddd}, then the primal-dual iterate $(\boldZ_{n})_{n}$ a.s. recurs in all the neighbors (w.r.t. $\norm{\cdot}_{*}$) of the unique variational Nash equilibrium $\bz_{*}\in\SOL(\X\times\real_{\geq 0}^{R},\tilde{\bv})$ of $\Gamma$, i.e., a.s. there exists a subsequence $(\boldZ_{n_{k}})_{k}$ of $(\boldZ_{n})_{n}$ which converges to $\bz_{*}$ w.r.t. $\norm{\cdot}_{*}$.
\end{lemma}
\begin{proof}
	Clearly since $\boldX_{n}\in\mathcal{X}$ for all $n\in\nat$ with $\X$ compact, the sequence $(\boldX_{n})_{n\in\nat}$ is non-explosive, i.e. $\norm{\boldX_{n}}<\infty$ for all $n\in\nat$. Moreover, via standard Gr\"{o}nwall's argumentation, one can infer that $\boldLambda_{n}$ is also non-explosive. As a consequence of these facts, we can assume, without loss of generality, that \eqref{Eq:aooajsjsjskkjdjjjddd} holds for all $k\in \nat$.

Now, notice that by Theorem \ref{Thm:UpperBoundFench}, it follows that $\mathcal{V}_{n}=\mathcal{V}_{n}(\bz_{*})$ is upper bounded by:
	\begin{align}
	\tau_{n}\left(\tfrac{\sum_{k=0}^{n-1}\gamma_{k}\upeta_{k}}{\tau_{n}}+\tfrac{\tilde{C}_{1}\sum_{k=0}^{n-1}\gamma^{2}_{k}}{\tau_{n}}+\tfrac{\sum_{k=0}^{n-1}\gamma_{k}\Uppsi_{k}}{\tau_{n}}+\tfrac{\boldS_{n}+\tfrac{2}{K}\boldR_{n}}{\tau_{n}}\right)\label{Eq:aiaishhfggfhfhffff}
	\end{align}
	where $\tau_{n}:=\sum_{k=0}^{n-1}\gamma_{k}$, $\Uppsi_{k}:=\Uppsi_{k}(\boldLambda_{k},\boldlambda_{*})$, $\upeta_{k}:=\upeta_{k}(\bz^{*})$, and $\tilde{C}_{1}=2((C_{2}^2/K)+(C_{3}^{2}/2))$.
	Thus, by our choice of the augmentation function and the condition \eqref{Eq:aooajsjsjskkjdjjjddd}, it follows that $\Uppsi_{k}\leq\alpha_{k}\norm{\boldlambda_{*}}_{2}^{2}/2$ (see \eqref{Eq:AugFunc}). Setting this estimate into \eqref{Eq:aiaishhfggfhfhffff}, it follows that $\mathcal{V}_{n}/\tau_{n}$ is upper bounded by:
	\begin{align}
	\tfrac{\sum_{k=n_{0}}^{n-1}\gamma_{k}\upeta_{k}}{\tau_{n}}+\tfrac{\tilde{C}_{1}\sum_{k=n_{0}}^{n-1}\gamma^{2}_{k}}{\tau_{n}}+\tfrac{\norm{\boldlambda_{*}}_{2}^{2}\sum_{k=n_{0}}^{n-1}\gamma_{k}\alpha_{k}}{2\tau_{n}}+\tfrac{\boldS_{n}+\tfrac{2}{K}\boldR_{n}}{\tau_{n}}.\label{Eq:aiaishhfggfhfhffffak}
	\end{align}
	Now, Lemma \ref{Lem:MartNoise} asserts that both $(\boldS_{n})_{n}$ and $(\boldR_{n})_{n}$ converge a.s. and therefore a.s. there exists a constant $A$ s.t. $\boldS_{n}+(2/K)\boldR_{n}\leq A$ for all $n$. This and \eqref{Eq:aiaishhfggfhfhffffak} yield that $\mathcal{V}_{n}/\tau_{n}$ is upper bounded by:
	\begin{align}
	\tfrac{\sum_{k=n_{0}}^{n-1}\gamma_{k}\upeta_{k}}{\tau_{n}}+\tfrac{\tilde{C}_{1}\sum_{k=n_{0}}^{n-1}\gamma^{2}_{k}}{\tau_{n}}+\tfrac{\norm{\boldlambda_{*}}_{2}^{2}\sum_{k=n_{0}}^{n-1}\gamma_{k}\alpha_{k}}{2\tau_{n}}+\tfrac{A}{\tau_{n}}.\label{Eq:aiaishhfggfhfhffff2}
	\end{align}
	
	For the final step of the proof, let $U$ be an arbitrary neighborhood of $\bz_{*}$ w.r.t. the norm on $\V\times\real^{R}$. Moreover, suppose that $\boldZ_{n}\notin U$ for all sufficiently large $n\geq 0$. W.l.o.g. we assume that $\boldZ_{n}\notin U$ for all $n\geq 0$. The latter assumption, and the fact $\tilde{\bv}$ strictly monotone, which is a consequence of strict monotonicity of $\bv$, and \eqref{Eq:StrongMon},  
	assert that we can find $c>0$ s.t. $\upeta_{k}\leq -c$, $\forall k\geq 0$.
	Consequently, we have the following upper bound of $\mathcal{V}_{n}$:
	\begin{equation}
	\label{Eq:aoaojdhhddjeeerrr}
\tau_{n}\left(-c+\tilde{C}_{1}\tfrac{\sum_{k=n_{0}}^{n-1}\gamma^{2}_{k}}{\tau_{n}}+\tfrac{\sum_{k=n_{0}}^{n-1}\gamma_{k}\alpha_{k}}{2\tau_{n}}\norm{\boldlambda_{*}}_{2}^{2}+\tfrac{A}{\tau_{n}}\right).
	\end{equation}
	Now, \eqref{Eq:ClosedLoop} and \eqref{Eq:aoaoshshsjjhdggdhhss2} assert that:
	\begin{equation}
	\label{Eq:hajhsjsjskkdjjdd}
	\tilde{C}_{1}\tfrac{\sum_{k=n_{0}}^{n-1}\gamma^{2}_{k}}{\tau_{n}}+\tfrac{\norm{\boldlambda_{*}}_{2}^{2}\sum_{k=n_{0}}^{n-1}\gamma_{k}\alpha_{k}}{2\tau_{n}}+\tfrac{A}{\tau_{n}}\rightarrow 0,\quad\text{as }n\rightarrow\infty.
	\end{equation}
	Finally, since $\tau_{n}\rightarrow\infty$ as $n\rightarrow\infty$, we have from \eqref{Eq:hajhsjsjskkdjjdd} and \eqref{Eq:aoaojdhhddjeeerrr} that a.s. $\mathcal{V}_{n}\rightarrow-\infty$, $n\rightarrow\infty$. This contradicts with the fact that $\mathcal{V}_{n}\geq-\tilde{\rmF}(\bz_{*},\boldZ_{0})>-\infty$. Thus, we obtain as desired that a.s. $\boldZ_{n}\in U$ for infinitely many $n\geq 0$. 
\end{proof}
The final auxiliary step that we need for showing Theorem \ref{Thm:aoasjsjsjdhdjshshs}, is the following result in the spirit of Robbin-Siegmund's Lemma (see e.g., Thm. 1.3.12 in \cite{Duflo1997}):  
\begin{lemma}
	\label{Lem:aaosjjshhdhdhdhd}
	Suppose that the Assumptions \ref{Ass:aiishhfggfjhhdjjdhhdjjddd}, \ref{Ass:aiishhfggfjhhdjjdhhdjjddd2}, \ref{Eq:aioaoosjjshhdjjddss}, and \ref{Ass:aoasjsdhhggehegeeeerrrttt} hold, and that $(\gamma_{n})_{n\in\nat_{0}}$, $(\bM_{n})_{n\in\nat}$, and $(\alpha_{n})_{n\in\nat_{0}}$ fulfill \eqref{Eq:oaaosjsjdhpllalaala}  \eqref{Eq:ClosedLoop}, \eqref{Eq:aoaoshshsjjhdggdhhss2}, and \eqref{Eq:aooajsjsjskkjdjjjddd}.
	Let $\bz_{*}$ be a variational Nash equilibrium, i.e. solution of $\bz_{*}\in\VI(\X\times\real_{\geq 0}^{R},\tilde{\bv})$. Then a.s. $(\tilde{\rmF}(\bz_{*},\boldZ_{n}))_{n}$ converges to a finite RV.		
\end{lemma}
\begin{proof}
W.l.o.g. we can assume that	\eqref{Eq:aooajsjsjskkjdjjjddd} holds for all $n$. Now, we have by \eqref{Eq:aiaishhfggfhfhffffak}:
	\begin{equation}
	\begin{split}
	&\tilde{\rmF}(\bz_{*},\boldZ_{n+1})
	\leq  \tilde{\rmF}(\bz_{*},\boldZ_{n})+\gamma_{n}\upeta_{n}+\tfrac{\gamma_{n}\alpha_{n}\norm{\boldlambda_{*}}_{2}^{2}}{2}\\
	&+\gamma_{n}^{2}\tilde{C}_{1}+\gamma_{n}\tilde{\bM}_{n}+\gamma_{n}^{2}\norm{\bM_{n+1}}^{2}_{*},
	\end{split}
	\label{Eq:aoaosjsjsdjdjdhfgfhgfhdhdgfgf}
	\end{equation}
	
	To continue, notice that $\boldZ_{n}$ is $\mathcal{F}_{n}$ measureable and $\upeta_{n}=\upeta_{n}(\bz^{*})\leq 0$ for all $n$ (this follows from the assumptions that $\bz^{*}$ is the solution of $\VI(\X\times\real_{\geq 0}^{R},\tilde{\bv})$ and $\tilde{\bv}$ is monotone), and $\Erw_{n}[\tilde{\bM}_{n}]=0$. Therefore, taking the expectation of \eqref{Eq:aoaosjsjsdjdjdhfgfhgfhdhdgfgf} given $\mathcal{F}_{n}$, we obtain: 
	\begin{align}
	\Erw[\tilde{\rmF}(\bz_{*},\boldZ_{n+1})|\mathcal{F}_{n}]
	\leq  \tilde{\rmF}(\bz_{*},\boldZ_{n})+\zeta_{n},
	\label{Eq:aoaosjsjshdhdggfhff}
	\end{align}
where $\zeta_{n}:=\tfrac{\gamma_{n}\alpha_{n}\norm{\boldlambda_{*}}_{2}^{2}}{2}+\gamma_{n}^{2}\tilde{C}_{1}+\gamma_{n}^{2}\Erw[\norm{\bM_{n+1}}^{2}_{*}|\mathcal{F}_{n}]$. For further step, Notice that by the assumptions \eqref{Eq:ClosedLoop}, \eqref{Eq:aoaoshshsjjhdggdhhss2}, \eqref{Eq:oaaosjsjdhpllalaala}, and Remark \ref{Rem:aaauiusgsghsgshshsddd}, $(\zeta_{n})_{n}$ is a non-negative $\mathbb{F}$-adapted sequence which fulfills $\sum_{k=0}^{\infty}\zeta_{k}<\infty$ a.s..

	Now, we aim now to exploit \eqref{Eq:aoaosjsjshdhdggfhff}. For this sake, consider the sequence $\boldY_{n}:=\tilde{\rmF}(\bz_{*},\boldZ_{n})-\sum_{k=0}^{n-1}\zeta_{k}$. 
Clearly, $(\boldY_{n})$ is $\mathbb{F}$-adapted and integrable. Moreover, the fact that $(\zeta_{n})$ is $\mathbb{F}$-adapted gives $\Erw_{n}[\boldY_{n+1}]=\Erw_{n}[\tilde{\rmF}(\bz_{*},\boldZ_{n+1})]-\sum_{k=0}^{n}\zeta_{k}$.
	So combining this and \eqref{Eq:aoaosjsjshdhdggfhff}, we have $\Erw_{n}[\boldY_{n+1}]\leq \boldY_{n}$,
	and consequently $(\boldY_{n})$ is a supermartingale. Notice that this supermartingale fulfills $\boldY_{n}\geq -\sum_{k=0}^{\infty}\zeta_{k}>-\infty$, for all $n\geq 0$.
	Therefore, the Martingale Convergence Theorem asserts that $(\boldY_{n})$ a.s. converges. Finally, since $0\leq\sum_{k=0}^{\infty}\zeta_{k}<\infty$ a.s., it follows that $(\tilde{\rmF}(\bz_{*},\boldZ_{n}))_{n}$ converges a.s., as desired.
\end{proof}
Now, we are ready to prove our main result:
\begin{proof}[Proof of Theorem \ref{Thm:aoasjsjsjdhdjshshs}]
	Notice first that Lemma \ref{Lem:aoaosjsjshdhdjshssjjddd} asserts (a.s.) the existence of a subsequence $(\boldZ_{n_{k}})_{k}$ of $(\boldZ_{n})_{n}$, which converges to $\bz_{*}$ the unique solution of $\VI(\X\times\real_{\geq 0}^{R},\tilde{\bv})$. Furthermore, the reprocity condition asserts that $\rmF(\bz_{*},\boldZ_{n_{k}})\rightarrow 0$. Next, Lemma \ref{Lem:aaosjjshhdhdhdhd} asserts that $(\rmF(\bz_{*},\boldZ_{n}))_{n}$ converges a.s. to a finite random variable and therefore a.s. a Cauchy sequence. As a consequence of the previous observations, for any $\epsilon>0$, a.s., we can choose $n_{0}$ large enough s.t. for all $n_{k},n,\tilde{n}\geq n_{0}$, $\rmF(\bz_{*},\boldZ_{n_{k}})\leq \epsilon/2$ and $\rmF(\bz_{*},\boldZ_{\tilde{n}})\leq \rmF(\bz_{*},\boldZ_{n})+ \epsilon/2$.
	So for the latter event, we have for $n,n_{k}\geq n_{0}$, $\rmF(\bz_{*},\boldZ_{n})\leq \rmF(\bz_{*},\boldZ_{n_{k}})+ \epsilon/2\leq \epsilon$.
	Thus a.s. $\rmF(\bz_{*},\boldZ_{n})\rightarrow 0$ as $n\rightarrow\infty$ and Proposition \ref{Prop:aaisshhfjffjfjfff} asserts that $(\boldZ_{n})$ converges a.s. to $\bz_{*}=(\bx_{*},\boldlambda_{*})$ as $n\rightarrow\infty$. For the desired statement, notice that by Proposition \ref{Prop:aiaisshshjdhdggddd3}, $\bx^{*}$ is the unique variational Nash equilibrium.  
\end{proof}
\section{Resource Constraint Violation Analysis}
\label{Sec:aiaiishshggsshgshsss}
Theorem \ref{Thm:aoasjsjsjdhdjshshs} asserts that in the long term and under suitable choices of algorithm parameters, the population iterate a.s. satisfies the coupled resource constraint. However, the guarantee is purely asymptotic, i.e., we only know that for large times, the population's iterate almost fulfills the resource constraints. Motivated to eliminate this drawback, we investigate in this section the non-asymptotic behavior of the iterate guarantee. In particular, we are interested in how the amount of resource constraint violation of MAARP decays with time. To make the corresponding result more accessible, we assume in this section the following:
\begin{assum}
	\label{Assum:aaishshdgdgdhdgdgd}
For all $n\in\nat_{0}$, $\alpha_{n}=\alpha\gamma_{n}$ for an $\alpha>0$.
The trackability condition \eqref{Eq:aooajsjsjskkjdjjjddd} holds for all $n\in\nat$, which is in this context:
\begin{equation}
\label{Eq:aaosjsjdhdhdgdgfgfgfgfff}
\alpha^{2}\gamma_{n}-\tfrac{\alpha}{2}+\tfrac{C_{1}^{2}}{K}\leq 0,\quad\forall n\in\nat_{0}.
\end{equation}
The noise is persistent in the sense that there exists $\sigma>0$ s.t. $\sigma_{k}\leq \sigma$ for all $k\in\nat$.
\end{assum}

Our measure for the resource constraint violation of MAARP is the following:   
\begin{equation*}
\text{CVio}_{n}^{r}:=\tfrac{\sum_{k=0}^{n-1}\gamma_{k}\bg_{r}(\boldX_{k})}{\sum_{k=0}^{n-1}\gamma_{k}}.
\end{equation*}
Since $\bg_{r}$ is convex for every $r\in [R]$, it follows from Jensen's inequality that $\text{CVio}_{n}^{r}$ is an upper bound of the constraint violation on the resource $r$ caused by the ergodic average of the population iterate at time $n-1$, i.e.:
\begin{equation}
\label{Eq:aiaiasshsgdgddhsgsgsgdffddd}
\bg_{r}(\overline{\boldX}_{n})\leq \CVio_{n}^{r},\quad\text{where}~\overline{\boldX}_{n}:=\tfrac{\sum_{k=0}^{n-1}\gamma_{k}\boldX_{k}}{\sum_{k=0}^{n-1}\gamma_{k}}.
\end{equation}

We start with the following auxiliary statement, which asserts that one can deduce the information about the load from the prices:
\begin{lemma} Suppose that $\boldLambda_{0}=0$.
	For all $r\in [R]$ and $n\in\nat$:
\begin{equation*}
\CVio_{n}^{r}\leq \tfrac{\norm{\boldLambda_{n}}_{2}+\sum_{k=0}^{n-1}\gamma_{k}\alpha_{k}\norm{\boldLambda_{k}}_{2}}{\sum_{k=0}^{n-1}\gamma_{k}}.
\end{equation*}
\end{lemma}
\begin{proof}
The definition of our price policy gives $\boldLambda^{r}_{k+1}\geq\boldLambda^{r}_{k}+\gamma_{k}\bg_{r}(\boldX_{k})-\gamma_{k}\alpha_{k}\boldLambda^{r}_{k}$.
So by summing this inequality and subsequent telescoping, we have, since $\boldLambda_{0}=0$, $\sum_{k=0}^{n-1}\gamma_{k}\bg_{r}(\boldX_{k})\leq \boldLambda_{n}^{r}+\sum_{k=0}^{n-1}\gamma_{k}\alpha_{k}\boldLambda_{k}^{r}$. Now, this inequality
and the fact that $\boldLambda^{r}_{k}\geq 0$ give $\boldLambda_{k}^{r}\leq\sqrt{\sum_{r=1}^{R}(\boldLambda_{k}^{r})^{2}}=\norm{\boldLambda_{k}}_{2}$.
Therefore as desired $\sum_{k=0}^{n-1}\gamma_{k}\bg_{r}(\boldX_{k})\leq \norm{\boldLambda_{n}}_{2}+\sum_{k=0}^{n-1}\gamma_{k}\alpha_{k}\norm{\boldLambda_{k}}_{2}$.
\end{proof}
Therefore to show the decay $\text{CVio}_{n}^{r}$, we only need to investigate the growth of the prices:  
\begin{lemma}
Suppose that $\boldLambda_{0}=0$ and $\boldY_{0}=0$. Then under Assumption \ref{Assum:aaishshdgdgdhdgdgd}, it holds: 
\begin{equation*}
\Erw[\norm{\boldLambda_{n}}_{2}]\leq \mathcal{O}\left((1+\sigma)\sqrt{\sum_{k=0}^{n-1}\gamma_{k}^{2}} \right) .
\end{equation*}	
\end{lemma}
\begin{proof}
By taking the expectation of \eqref{Eq:aoaosjsjshdhdggfhff}, and subsequent summing and telescoping, we can conclude that $\Erw\left[ \tilde{\rmF}(\bz_{*},\boldZ_{n})\right] \leq \tilde{\rmF}(\bz_{*},\boldZ_{0})+\tfrac{\sum_{k=0}^{n-1}\gamma_{k}\alpha_{k}\norm{\boldlambda_{*}}_{2}^{2}}{2}+\sum_{k=0}^{n-1}\gamma_{k}^{2}\tilde{C}_{1}
+\sum_{k=0}^{n-1}\gamma_{k}^{2}\sigma_{k+1}^{2}$.
Consequently, since $\boldZ_{0}=0$:
\begin{align*}
\Erw[\tfrac{\norm{\boldLambda_{n}-\boldlambda_{*}}^{2}}{2}]
&\leq \Delta\uppsi+\tfrac{\norm{\boldlambda_{*}}_{2}^{2}}{2}+\tfrac{\sum_{k=0}^{n-1}\gamma_{k}\alpha_{k}\norm{\boldlambda_{*}}_{2}^{2}}{2}+\sum_{k=0}^{n-1}\gamma_{k}^{2}\tilde{C}_{1}\\
&+\sum_{k=0}^{n-1}\gamma_{k}^{2}\sigma_{k+1}^{2},
\end{align*}
where:
\begin{equation}
\label{Eq:aiaishshhdhddgdgdhdgddd}
\Delta\uppsi=\sum_{i=1}^{N}\left[ \max_{\bx^{(i)}\in\X_{i}}\uppsi_{i}(\bx^{(i)})-\min_{\bx^{(i)}\in\X_{i}}\uppsi_{i}(\bx^{(i)})\right].
\end{equation}
Thus by Jensen's inequality and triangle inequality:
\begin{align*}
\left( \Erw[\norm{\boldLambda_{n}}_{2}]-\norm{\boldlambda_{*}}_{2}\right) ^{2}&\leq 2 \Delta\uppsi+\norm{\boldlambda_{*}}_{2}^{2} +\sum_{k=0}^{n-1}\gamma_{k}\alpha_{k}\norm{\boldlambda_{*}}_{2}^{2}\\&+2\sum_{k=0}^{n-1}\gamma_{k}^{2}\tilde{C}_{1}
+2\sum_{k=0}^{n-1}\gamma_{k}^{2}\sigma_{k+1}^{2}.
\end{align*}
The desired statement follows from above since $\alpha_{k}=\alpha \gamma_{k}$ and $\sigma_{k}\leq \sigma$ for all $k$.  
\end{proof}
Now, if $\gamma_{k}=\Theta(1/\sqrt{k})$, we have $\Erw[\norm{\boldLambda_{n}}_{2}]\leq\mathcal{O}((1+\sigma)\sqrt{\ln(n)})$.
Moreover, since:
\begin{align*}
\sum_{k=0}^{n-1}\gamma_{k}\alpha_{k}\Erw[\norm{\boldLambda_{k}}_{2}]&\leq\alpha\sum_{k=0}^{n-1}\gamma_{k}^{2}\mathcal{O}(\sqrt{\ln(k)})\\
&\leq\mathcal{O}(\sqrt{\ln(n)})\sum_{k=0}^{n-1}\gamma_{k}^{2},
\end{align*}
it holds $\Erw[\text{CVio}^{r}_{n}]\leq \mathcal{O}(\ln^{3/2}(n)/\sqrt{n})$.
This observation and \eqref{Eq:aiaiasshsgdgddhsgsgsgdffddd} immediately give the following non-asymptotic result for the violation of the coupled constraint:
\begin{theorem}
\label{thm:ErgConv}
Suppose that $\gamma_{k}=\Theta(1/\sqrt{k})$, then for every $r\in [R]$, we have the guarantee $\Erw[\bg_{r}(\overline{\boldX}_{n})]\leq \mathcal{O}(\ln^{3/2}(n)/\sqrt{n})$ 
\end{theorem} 

\begin{remark}
	\label{Rem:aaiisgsgsgsffsdsdss}
By choosing a mirror map different than the Euclidean projection, one can obtain a bound in Theorem \ref{thm:ErgConv} that has weaker dimensional dependence. 
 For instance, consider the case where $\mathcal{X}_{i}$ is a $D_{i}$-dimensional simplex and the case where the noise is a sequence of independent Gaussian normal vector in $\real^{D}$. In the case where, $\norm{\cdot}_{i}$ is the $2$-norm dual to itself,  $\uppsi_{i}$ is the Euclidean norm, and the mirror maps are Euclidean projections, it holds  $\Delta\uppsi=\sum_{i=1}^{N}\sqrt{D_{i}}$ and $\sigma^{2}=\sum_{i=1}^{N}D_{i}$. In contrast, we have in case $\norm{\cdot}_{i}$ is the $1$-norm dual whose dual is the $\infty$-norm, $\uppsi_{i}$ is the Gibbs entropy, and the mirror maps are logit choices, it holds $\Delta\uppsi=\sum_{i=1}^{N}\ln(D_{i})$ and $\sigma^{2}\leq \sum_{i=1}^{N}(2(\sqrt{\ln(D_{i})}+\ln(D_{i}))+1)$, where $C>0$ is a universal constant (for the latter see Example 2.7 in \cite{Boucheron2013}).
\end{remark}
\section{Convergence Analysis for Ergodic Average}
\label{Sec:iaaiiaahsgshgshhsssss}
Theorem \ref{thm:ErgConv} guarantees the decaying coupled constraints violation in expectation of the ergodic average of the population actions for step size $\gamma_{k}$ and augmentation sequence $\alpha_{k}$ of order $\Theta(1/\sqrt{k})$. In contrast, Theorem \ref{Thm:aoasjsjsjdhdjshshs} does not ensure\footnote{The convergence of the iterate implies the convergence of the ergodic average of the iterates} the a.s. fulfillment of the ergodic average in the asymptotic regime for the step size sequences of order $\Theta(1/\sqrt{k})$. To close this theoretical gap, we show the convergence of the ergodic average for that class of step sizes:
\begin{theorem}
	\label{Thm:aiaishshgdgdgdffsggss}
Suppose that $v$ is strictly monotone, and
suppose that there exists $\sigma>0$ s.t. $\sigma_{k}\leq \sigma$, for all $k\in\nat$, and that:
\begin{align}
&\sum_{k=0}^{\infty}\gamma_{k}=\infty\quad\tfrac{\sum_{k=0}^{n-1}\gamma^{2}_{k}}{\sum_{k=0}^{n-1}\gamma_{k}}\rightarrow 0\label{Eq:aasosjdhdhdggdshdggddhd}\\
&\tfrac{\sum_{k=0}^{n-1}\gamma_{k}\alpha_{k}}{\sum_{k=0}^{n-1}\gamma_{k}}\rightarrow 0\label{Eq:aiaiahshshsggs}\\
&\sum_{k=0}^{\infty}\tfrac{\gamma_{k}^{2}}{\left( \sum_{i=0}^{k}\gamma_{i}\right)^{2}}<\infty. \label{Eq:aiaiahshshsggs2}
\end{align}	
Then it holds $\lim_{n\rightarrow\infty}\overline{\boldX}_{n}=x_{*}$
\end{theorem}
Before we provide the proof of the above theorem, let us first have a discussion about the step size and the augmentation sequence that fulfill the above requirement:
\begin{remark}
Clearly, a step size sequence of order $\gamma_{k}=\Theta(1/k^{p})$, where $p\in(0,1]$ fulfills \eqref{Eq:aasosjdhdhdggdshdggddhd} and \eqref{Eq:aiaiahshshsggs2}. 
Choosing $\alpha_{k}=\alpha\gamma_{k}$ for an $\alpha>0$, one sees that \eqref{Eq:aiaiahshshsggs} is also fulfilled. 
\end{remark}

\begin{lemma}
	\label{Lem:aaisishhdhdgdggdhsgsgsdd}
	Suppose that:
	\begin{equation*}
	\sum_{k=0}^{\infty}\gamma_{k}=\infty\quad \sum_{k=0}^{\infty}\tfrac{\gamma_{k}^{2}\sigma_{k}^{2}}{\left( \sum_{i=0}^{k}\gamma_{i}\right)^{2}}<\infty
	\end{equation*}
	\begin{equation*}
    \tfrac{\boldS_{n}(\bx)}{\sum_{k=0}^{n-1}\gamma_{k}}\rightarrow 0\quad \text{a.s. }k\rightarrow\infty
	\end{equation*}
\end{lemma}
\begin{proof}
By the Cauchy-Schwarz inequality, by the assumption that $\mathcal{X}$ is compact, and by a similar argumentation as in Remark \eqref{Rem:aaauiusgsghsgshshsddd}, it follows that a.s. $\sum_{k=0}^{\infty}\tfrac{\gamma_{k}^{2}\Erw_{k}[\abs{\tilde{\bM}_{k+1}}^{2}]}{\left( \sum_{i=0}^{k}\gamma_{i}\right)^{2}}<\infty$. 
Finally, since $(\boldS_{n}(\bx))_{n}$ is a martingale (see Lemma \ref{Lem:MartNoise}), we have from Theorem 2.18. in \cite{Hall1980} the desired statement.
\end{proof}
\begin{lemma}
	\label{Lem:aaisishhdhdgdggdhsgsgsdd2}
Suppose that $\sigma_{k}\leq \sigma$, for all $k$, and that:
\begin{equation}
\label{Eq:aaiisjshdjjdhhdggshhsss}
\tfrac{\sum_{k=0}^{n-1}\gamma_{k}^{2}}{\sum_{k=0}^{n-1}\gamma_{k}}\rightarrow 0,\quad n\rightarrow\infty.
\end{equation}	
Then it holds:
\begin{equation*}
\tfrac{\boldR_{n}}{\sum_{k=0}^{n-1}\gamma_{k}}\rightarrow 0,\quad n\rightarrow\infty,~\text{a.s.}
\end{equation*}
\end{lemma}
\begin{proof}
$(\norm{\bM_{n}}_{*}^{2})_{n}$ is a submartingale with $\sup_{k\in\nat}\Erw[\norm{\bM_{n}}_{*}^{2}]<\infty$. So, it follows from the Martingale Convergence Theorem (Thm. 2.5 in \cite{Hall1980}) the a.s. convergence of $(\norm{\bM_{n}}_{*}^{2})_{n}$ and thus a.s. the existence of $A>0$ s.t. $\sup_{n\in\nat}\norm{\bM_{n}}_{*}^{2}\leq A$. Combining this with \eqref{Eq:aaiisjshdjjdhhdggshhsss}, the result follows.
\end{proof}
\begin{figure}[htbp]
	\begin{center}
		\includegraphics[scale=0.7, trim={4cm 11.5cm 4cm 11.9cm}, clip]{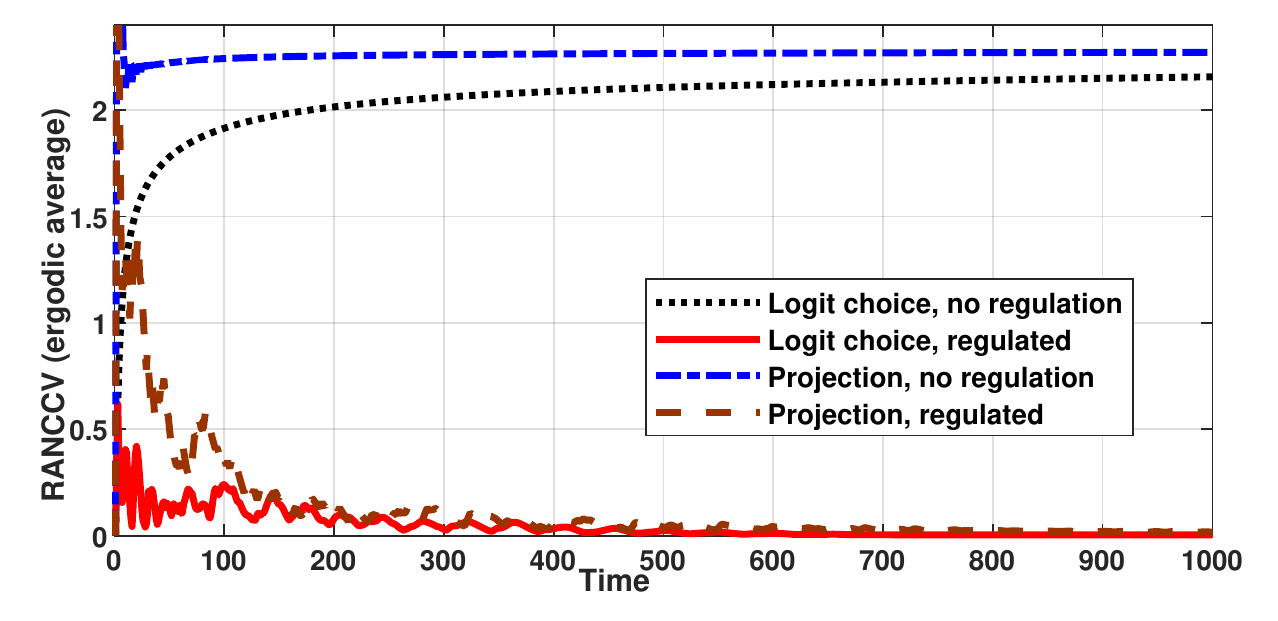}
	\end{center}
	\caption{RANCCV of the ergodic average of the population's dynamic for $D=R=20$, $N=50$, $d=10.5$, and $\sigma=0$.
	}
	\label{Fig:aoaojsjsjjddd1}
\end{figure}
\begin{proof}[Proof of Theorem \ref{Thm:aiaishshgdgdgdffsggss}]It holds:
\begin{equation*}	
\inn{\overline{\boldZ}_{n}-\bz_{*}}{\tilde{\bv}(\bz_{*})}=\tfrac{\sum_{k=0}^{n-1}\gamma_{k}\inn{\boldZ_{k}-\bz_{*}}{\tilde{\bv}(\bz_{*})}}{\sum_{k=0}^{n-1}\gamma_{k}}\geq\tfrac{\sum_{k=0}^{n-1}\gamma_{k}\upeta_{k}}{\sum_{k=0}^{n-1}\gamma_{k}},
\end{equation*}
where the inequality follows from the monotonicity of $\tilde{\bv}$. Thus, it follows from \eqref{Eq:aiaishhfggfhfhffffak} and the fact that $\mathcal{V}_{n}\geq \tilde{\rmF}(\boldZ_{0},z_{*})$:
\begin{align*}
\inn{\overline{\boldZ}_{n}-\bz_{*}}{\tilde{\bv}(\bz_{*})}&\geq-\tfrac{\tilde{\rmF}(\boldZ_{0},\bz_{*})}{\tau_{n}}-\tfrac{\tilde{C}_{1}\sum_{k=0}^{n-1}\gamma_{k}^{2}}{\tau_{n}}\\
&+\tfrac{\norm{\boldlambda_{*}}_{2}^{2}\sum_{k=n_{0}}^{n-1}\gamma_{k}\alpha_{k}}{2\tau_{n}}
+\tfrac{\boldS_{n}+\tfrac{2}{K}\boldR_{n}}{\tau_{n}}.
\end{align*}
Now, Lemma \ref{Lem:aaisishhdhdgdggdhsgsgsdd} and \ref{Lem:aaisishhdhdgdggdhsgsgsdd2} give $(\boldS_{n}+\tfrac{2}{K}\boldR_{n})/\tau_{n}\rightarrow 0$ a.s and consequently $\liminf_{n\rightarrow\infty}\inn{\overline{\boldZ}_{n}-\bz_{*}}{\tilde{\bv}(\bz_{*})}\geq 0$. Combining this statement with the fact that a.s. $\limsup_{n\rightarrow\infty}\inn{\overline{\boldZ}_{n}-\bz_{*}}{\tilde{\bv}(\bz_{*})}\leq 0$, which is implied by the fact that $\bz_{*}\in\SOL(\mathcal{X}\times\real^{R}_{\geq 0},\tilde{\bv})$, it follows that $\lim_{n\rightarrow\infty}\inn{\overline{\boldZ}_{n}-\bz_{*}}{\tilde{\bv}(\bz_{*})}= 0$.
Finally, since $\tilde{\bv}$ is strictly monotone, it follows that a.s. $\lim_{n\rightarrow\infty}\overline{\boldZ}_{n}=z_{*}$ as desired.
 \end{proof}
\section{Numerical Experiment}
\label{Sec:NumExp}
In this section, we numerically investigate the behavior of Algorithm \ref{Alg:aoaishhjddhhddddeee} in case the underlying game model is quadratic. This sort of games is popular in several applications (see e.g. \cite{Mertikopoulos2018,Paccagnan2017,Grammatico2017}) such as competitive markets, cognitive radio networks, charging of electric vehicles, and congestion control of road networks. Our focus here lies on the difference of the sustainable resource behavior, between the population's states and the their ergodic averages, and between different choices of the mirror map.  

\subsection{Setting}
\paragraph*{Exponential Weights Online Learning in the Quadratic Game}
\begin{figure}[htbp]
	\begin{center}
		\includegraphics[scale=0.7, trim={4cm 11.5cm 4cm 11.9cm}, clip]{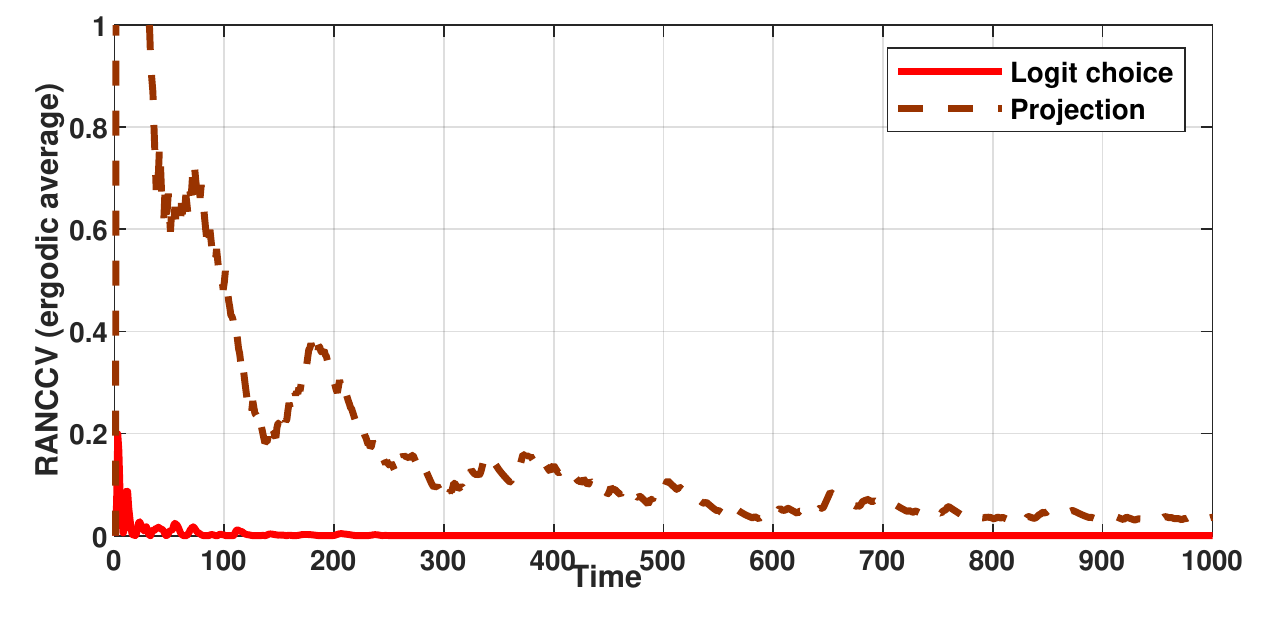}
	\end{center}
	\caption{RANCCV of the ergodic average of the population's dynamic for $D=R=50$, $N=100$, $d=8.5$, and $\sigma=0$.}
	\label{Fig:aoaojsjsjjddd2}
\end{figure}
We consider $N$ agents whose task is to allocate a certain amount of tasks to $R$ resources. The strategy space of agent $i$ corresponds to the simplex $\Delta:=\lrbrace{\bx^{(i)}\in\real^{R}_{\geq 0}:\sum_{r=1}^{R}\bx_{k}^{(i)}=1}$.
For a strategy, $\bx^{(i)}\in\Delta$, $x_{r}^{(i)}$ stands for the proportion of tasks agent $i$ assigns to resource $r\in [R]$. Moreover, we have, in this case, $D=R$. The cost function of player $i$ is quadratic and given by $$J^{(i)}(\bx^{(i)},\bx^{(-i)})=\tfrac{1}{2}\inn{\bx^{(i)}}{\bfQ \bx^{(i)}}+\inn{\bfC\sigma(\bx)+c^{i}}{\bx^{(i)}},$$ where $\sigma(\bx)=\tfrac{1}{N}\sum_{i=1}^{N}\bx^{(i)}$,
$c_{i}\in\real^{D}$, $\bfQ\in\real^{D\times D}$ and $\bfC\in\real^{D\times D}$ are positive semi-definite, and either $\bfQ$ or $\bfC$ are positive definite. In order to apply our method, we set $\ru_{i}(\bx)=-J^{(i)}(\bx)$. The corresponding gradient mapping is given by: $$\bv(\bx)=-\left[ (\mathbf{I}_{N}\otimes \bfQ+ \tfrac{1}{N}\mathbf{1}_{N}\mathbf{1}_{N}^{\T}\otimes C)\bx+c+\tfrac{1}{N}(\mathbf{I}_{N}\otimes C^{\T})\bx\right],$$
where $\otimes$ denotes the Kronecker product between two matrices, and where $\mathbf{1}_{N}$ denotes a vector in $\real^{N}$ whose entries are one and $\mathbf{I}_{N}$ denotes the identity matrix in $\real^{N\times N}$.

\paragraph*{Game Parameter}
\begin{figure}[htbp]
	\begin{center}
			\includegraphics[scale=0.7, trim={4cm 11.8cm 4.4cm 12cm}, clip]{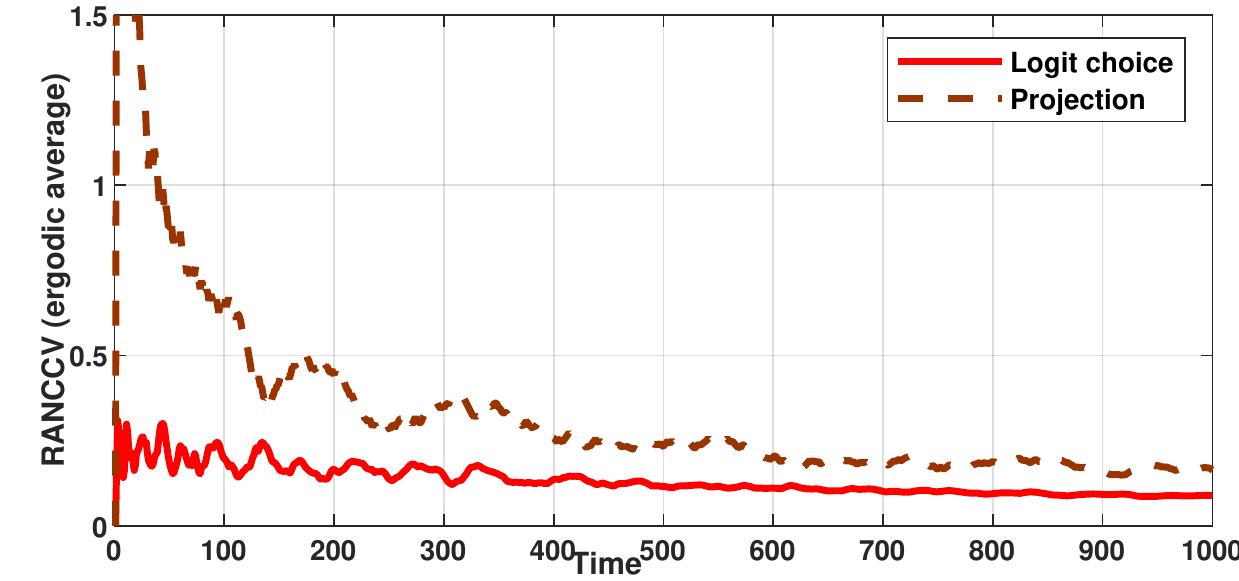}
	
	\end{center}
	\caption{RANCCV of the ergodic average of the population's dynamic for $D=R=50$, $N=100$, $d=8$, and $\sigma=0$}
	\label{Fig:aoaojsjsjjddd3}
\end{figure}
As the mirror map of the agents, we use either the Euclidean projection onto the simplex or the logit choice. In order to avoid numerical overflow in the implementation of the logit choice we use the log-sum trick. In the simulation, we set $\bfQ=2\sqrt{\tilde{\bfQ}^{\T}\tilde{\bfQ}}+\mathbf{I}_{D}$, where the entries of $\tilde{\bfQ}$ are chosen independently and normal distributed. Moreover we consider the case where $\bfC=4\mathbf{I}_{D}$ and $c=0$. We adapt as the model for resource constraints, the affine constraints as described in Remark \ref{Rem:aaiaishshshdjjdhdhdggsss}, with $\bfA=4\mathbf{I}_{D}$ and $b=d\mathbf{1}_{D}$ for different $d>0$. As the model for the stochastic feedback, we use the Gaussian vector with the covariance matrix $\sigma^{2}\mathbf{I}_{D}$, where either $\sigma=0$ or $\sigma=5$. Throughout the simulations, we choose the step size sequence $\gamma_{n}=0.5/\sqrt{n+1}$ and the augmentation sequence $\alpha_{n}=\alpha\gamma_{n}$ with $\alpha=5$. We plot the corresponding resource average negative clipped constraints violation (RANCCV), both of the ergodic average of the population's dynamic, i.e., $\sum_{r=1}^{R}[\bg_{r}(\overline{\boldX}_{n})]_{+}/R$, and of the population's dynamic, i.e., $\sum_{r=1}^{R}[\bg_{r}(\boldX_{n})]_{+}/R$.
\begin{figure}[htbp]
	\begin{center}
		\includegraphics[scale=0.7, trim={4cm 11.5cm 4cm 11.9cm}, clip]{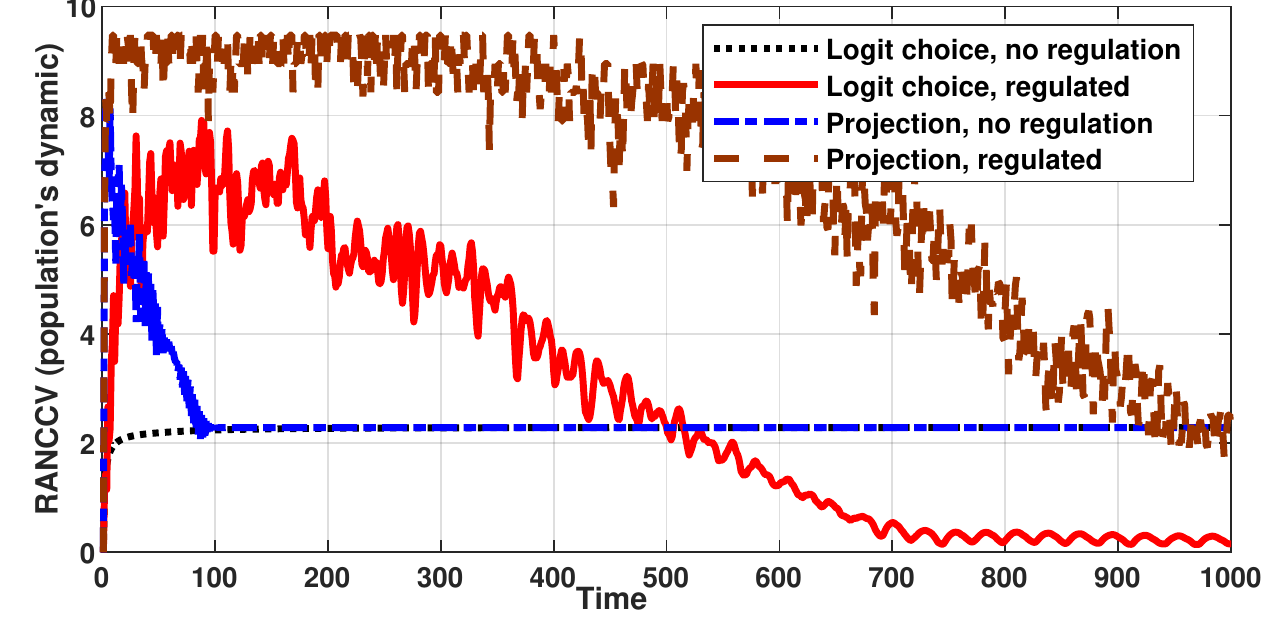}
	\end{center}
	\caption{RANCCV of the population's dynamic for $D=R=20$, $N=50$, $d=10.5$, and $\sigma=0$.}
	\label{Fig:aoaojsjsjjddd4}
\end{figure}
\begin{figure}[htbp]
	\centering
	\includegraphics[scale=0.7, trim={4cm 11.5cm 4cm 11.9cm}, clip]{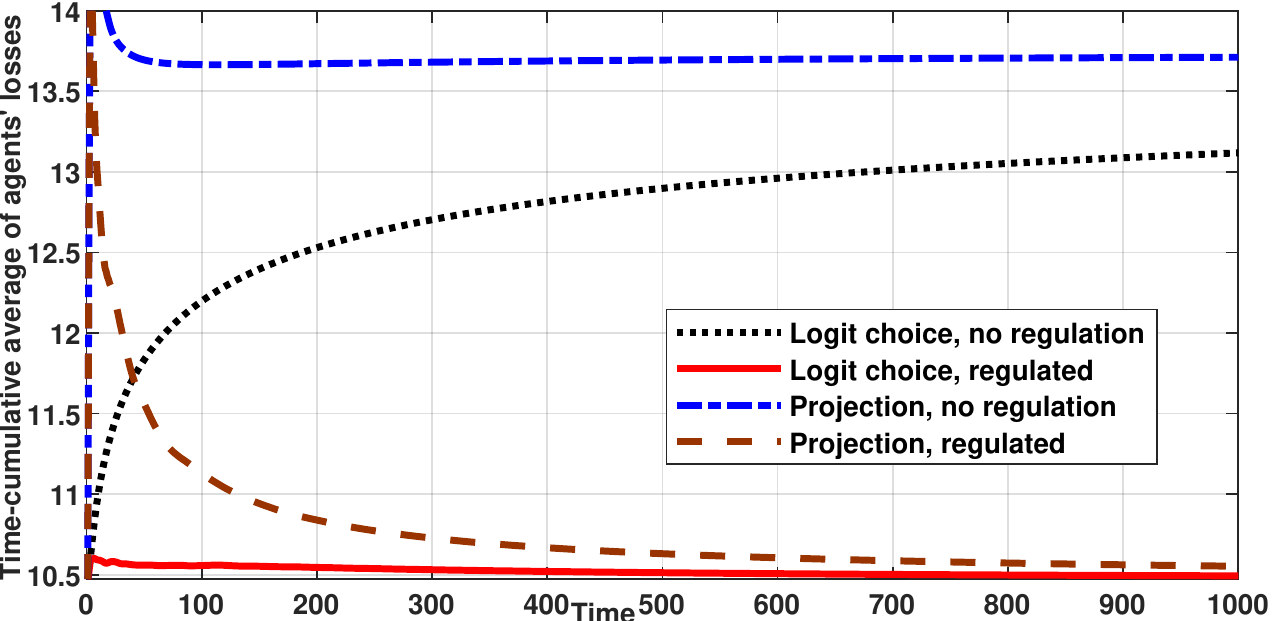}
	\caption{Comparison of Time-cumulative of average agents' losses for mirror ascent dynamic with (MAARP) and without pricing regulation.}
	\label{Fig:TimeCum}	
\end{figure}
\subsection{Evaluation}
\paragraph*{Price regulation vs. Anarchy - RANCCV}We first evaluate the results in case of no feedback noise, i.e., $\sigma=0$ (Figure \ref{Fig:aoaojsjsjjddd1}--\ref{Fig:aoaojsjsjjddd4}), in the case of no feedback noise, i.e., $\sigma=0$.
As apparent from Figure \ref{Fig:aoaojsjsjjddd1}, control of selfish agents is crucial for sustainable resource behaviour, since the average of the negative clipped resource constraints violation in the pure anarchistic case (black dotted - and blue dash-dotted line) \eqref{Eq:aoaaosjsdhdjdhhddss} is significantly higher than the case where MAARP is applied (red - and brown dashed line). We observe that better results yield if we use the logit choice (red line) instead of the usual Euclidean projection (brown dashed line). This effect is more pronounced in a high strategy space dimension than in the low one (cf. Figure \ref{Fig:aoaojsjsjjddd1} and Figures \ref{Fig:aoaojsjsjjddd2}--\ref{Fig:aoaojsjsjjddd3}), which is aligned with the discussion made in Remark \ref{Rem:aaiisgsgsgsffsdsdss}. 
\paragraph*{Price regulation vs. Anarchy - Utilities}
Given the previous discussion, one may think that the reduction of constraint violation comes with a reduction of the population's welfare. For this reason, we investigate the average of agents' losses, i.e. $\sum_{i=1}^{N}J^{(i)}(\overline{\boldX}_{n})/N$, numerically . We plot the time-average\footnote{The reason that we plot the time-average instead of the quantity itself is to make the performance distinction between MAARP with logit choice and with Euclidean projection clearer.} of this quantity in Figure \ref{Fig:TimeCum}. There, we observe that MAARP, in contrast to the pure anarchistic case, not only promotes sustainable behaviour, but also reduces the average loss, and therefore improves the population welfare. Furthermore, we see that the use of a mirror map other than Euclidean projection also improves not only sustainable behavior but also the population's welfare.

\paragraph*{Ergodic Average vs. Actual Trajectory} 
\begin{figure}[htbp]
	\begin{center}
		\includegraphics[scale=0.7, trim={4cm 11.8cm 4.4cm 12cm}, clip]{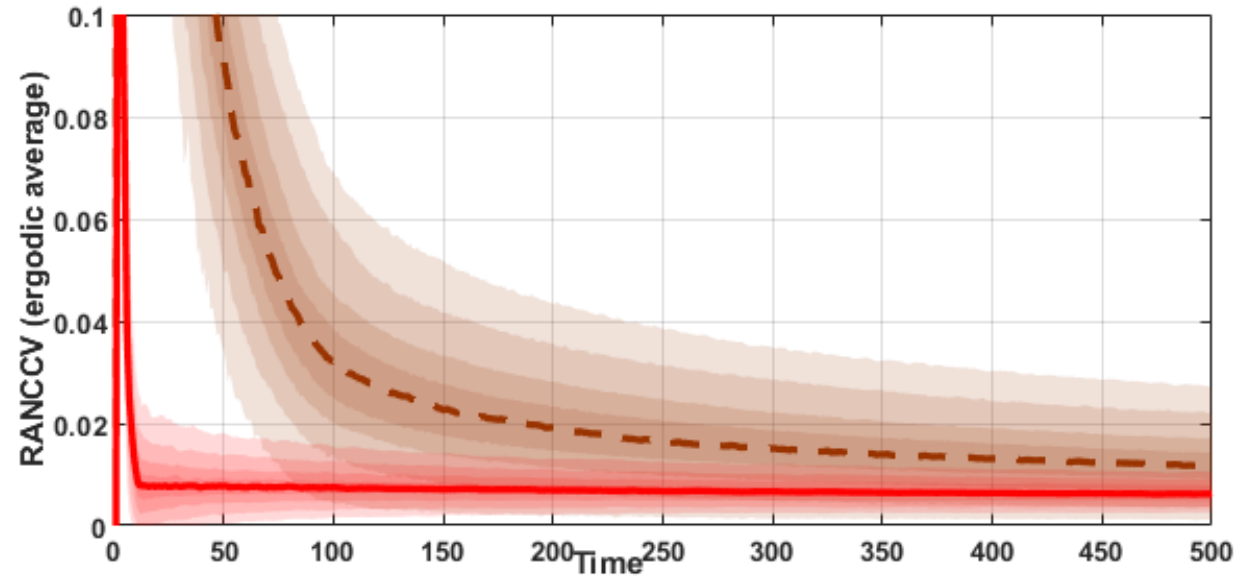}
	\end{center}
	\caption{RANCCV of the ergodic average of the population's dynamic for $D=R=20$, $N=50$, $d=10.5$, $\sigma=5$, and sample size $=500$. Red line corresponds to the sample average of RANCCV in the logit choice case and brown dashed line resp. in the Euclidean case. Shaded areas correspond to $25\%$-, $50\%$-, $75\%$-, and $90\%$-percentile.}
	\label{Fig:aoaojsjsjjddd5}
\end{figure}
\begin{figure}[htbp]
	\begin{center}
			\includegraphics[scale=0.7, trim={4cm 11.7cm 4.4cm 12cm}, clip]{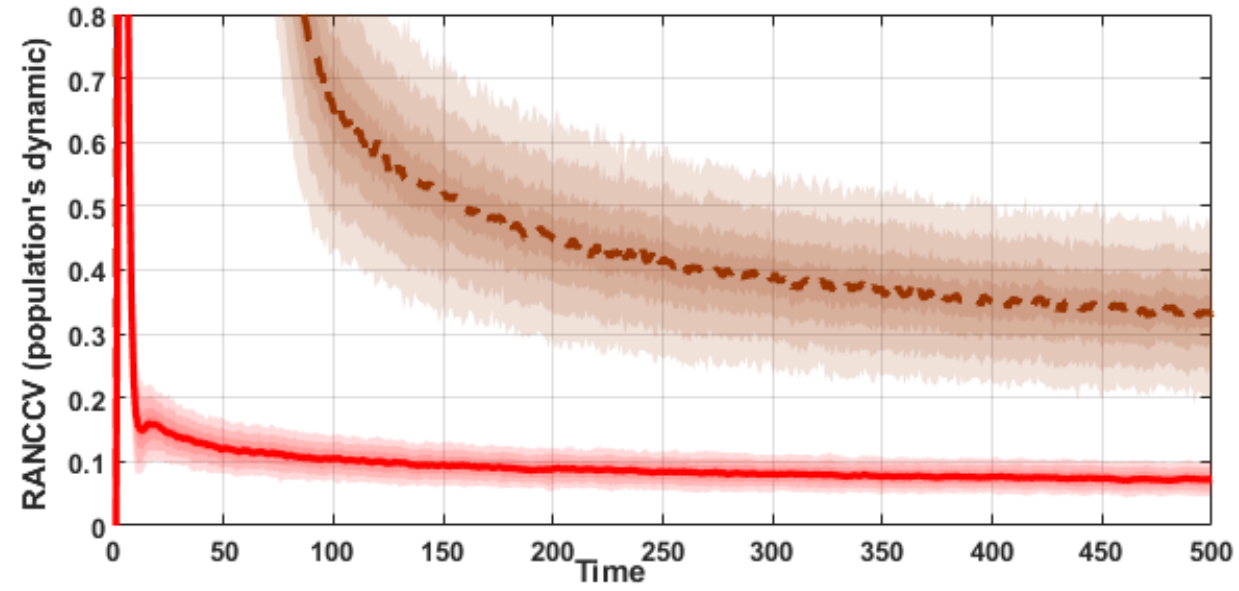}
	\end{center}
	\caption{RANCCV of the population's dynamic for $D=R=50$, $N=100$, $d=10.5$, $\sigma=5$, and sample size $=500$. Red line corresponds to the sample average of RANCCV in the logit choice case and brown dashed line resp. in the Euclidean case. Shaded areas correspond to $25\%$-, $50\%$-, $75\%$-, and $90\%$-percentile.}
	\label{Fig:aoaojsjsjjddd6}
\end{figure}
Figure \ref{Fig:aoaojsjsjjddd2} depicts the behavior of the actual population's dynamic. Therein we can observe the tendency of the decaying number of violations of the resource constraints. However, the corresponding decay might be much slower in comparison to decay of the ergodically time-averaged population's dynamic (see Figure \ref{Fig:aoaojsjsjjddd1}).

\paragraph*{Noise-Robustness of MAARP with Logit choice}Now, we evaluate the simulation results in case of feedback noise with $\sigma=5$ (Figures \ref{Fig:aoaojsjsjjddd5} and \ref{Fig:aoaojsjsjjddd6}) and $500$ noisy samples (i.i.d). One can see that using logit choice effects, in average, better results than using the Euclidean projection as forecasted by Remark \ref{Rem:aaiisgsgsgsffsdsdss}. Moreover, one can observe that the RANCCV is more volatile in the case where the Euclidean projection is used, compared to the case where the mirror map is the logit choice. Comparing Figures \ref{Fig:aoaojsjsjjddd5} and \ref{Fig:aoaojsjsjjddd6}, it is apparent that using the ergodic average yields a significantly lower and less volatile number of resource constraint violations.        

\paragraph*{Comparison to the State of the Art}
At last, we give a comparison of our method to some existing methods comparable to ours, i.e., the primal-dual method (see e.g., \cite{Facchinei1}) and the asymmetric projection (AP) (see e.g., Algorithm 2 \cite{Paccagnan2017}). The primal-dual method can be seen as the MAARP without augmentation (i.e., $\alpha=0$). Moreover, we give the primal-dual method a leverage by equipping it with the logit choice instead of the Euclidean projection. We plot the corresponding RANCCV in Figure \ref{Fig:RANCCVCComp}. There, one can see that the MARRP with logit choice (red line) outperforms the primal-dual method (yellow dash-dotted), and the best result yields if one uses AP (purple dotted). However, the dual variable update of AP requires, in contrast to MAARP, two consecutive congestion states of the resources and can not be implemented in parallel with the population's strategy update (see Remark \eqref{Rem:aaiishssggdhddgdggdgdd}). Besides this fact, one can see from the plot of the average of agents' losses in Figure \ref{Fig:AVUtil}, that the excellent performance of RANCCV of AP comes with the increase of agents' losses.
\begin{figure}[htbp]
	\centering
	\includegraphics[scale=0.7, trim={4cm 11.9cm 4.4cm 12cm}, clip]{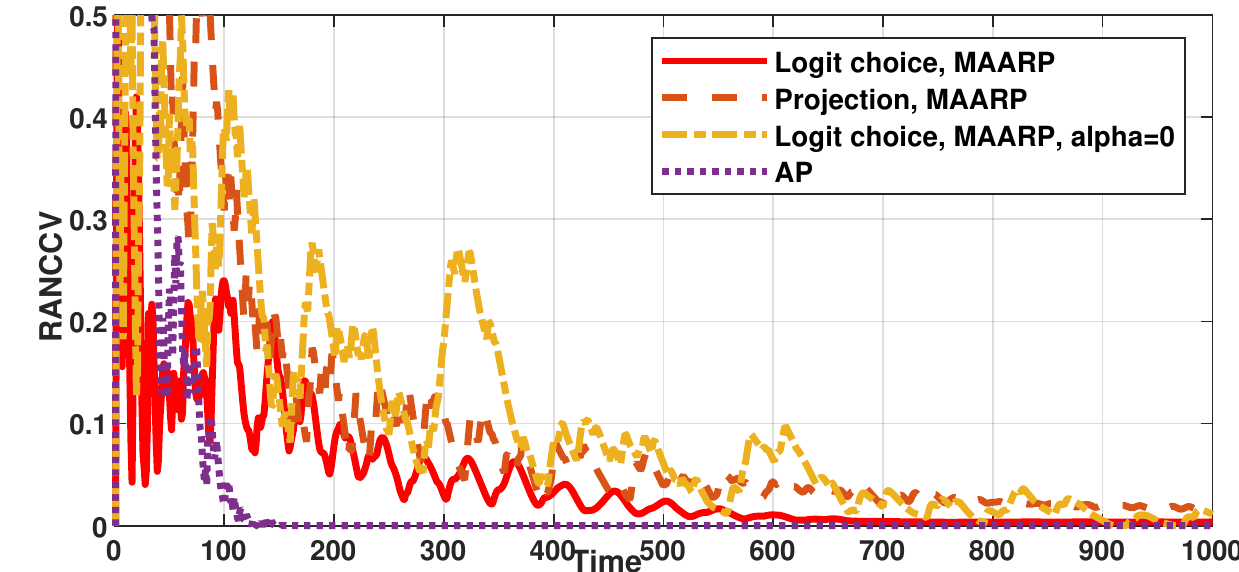}
	\caption{Comparison of RANCCV of MAARP with the state of the art.}
	\label{Fig:RANCCVCComp}	
\end{figure}

\begin{figure}[htbp]
	\centering
		\includegraphics[scale=0.7, trim={4cm 11.9cm 4.4cm 12cm}, clip]{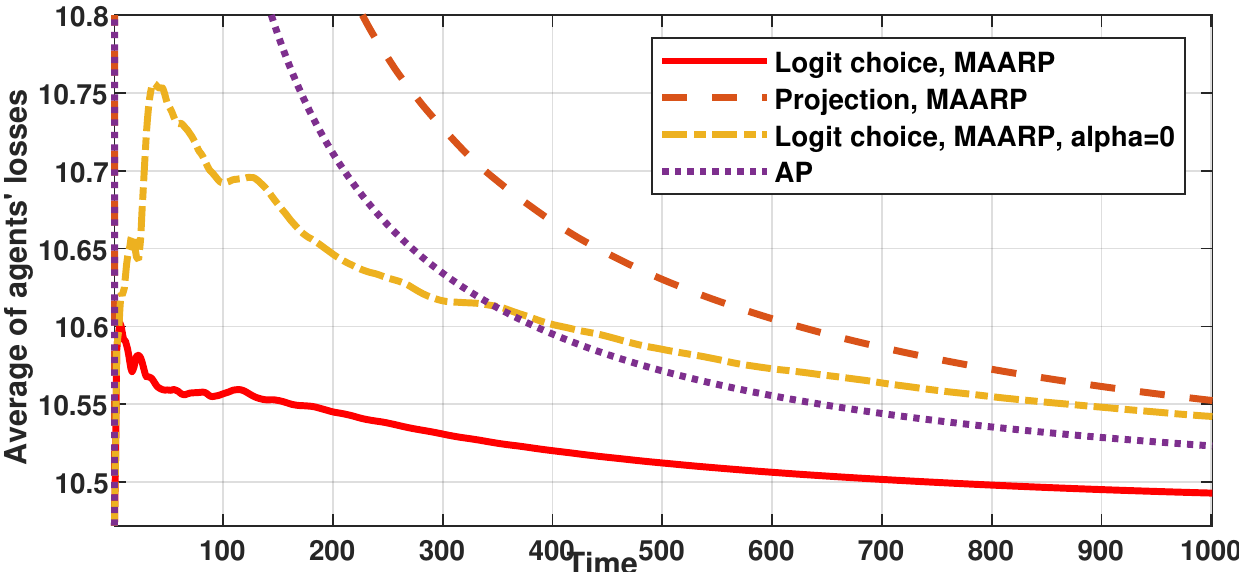}
	\caption{Comparison of the average agents' losses of MAARP with the state of the art.}
	\label{Fig:AVUtil}	
\end{figure}
\section{Summary, Discussion, Outlook, and Future Work}
In this work, we have proposed a novel decentralized pricing method that aims to encourage resource sustainable behavior in a population of selfish online learning agents having noisy first-order feedback by leading them toward a generalized Nash equilibrium of the game with corresponding coupled constraints.

The given results are based on the assumption that the utility functions and the constraint violation functions are continuously differentiable and that the utility function is strictly convex. However, it is straightforward to generalize the results in order to handle (not-necessarily continuously differentiable) convex function by replacing the gradients with subgradients, and by involving, in addition, the notion of convergence to a set (see \cite{Mertikopoulos2018}). The latter is necessary since the set of variational Nash equilibria $\SOL(\mathcal{Q},\bv)$ in the convex utility case is, in general, not a singleton. The simplification made in this work is only for the sake of readability. 

In the case that the martingale noise is persistent, the method gives a.s. convergence of the population's state to the generalized Nash equilibrium and consequently a.s. compliance of the resource constraints in the asymptotic limit for the rich class of polynomially decaying step-size policies of order $\gamma_{n}=\Theta(n^{-p})$ where $p\in(0,1]$. However, we were only able to give this guarantee for the ergodic average of the population iterates. For the actual population's iterate, we only could show the a.s. convergence for $p\in (1/2,1]$. Thus the case $p\in (0,1/2]$ remains open.

In the case where $\gamma_{n}=\Theta(n^{-1/2})$, we were able to provide in the persistent noise case, a non-asymptotic decaying expectation bound of order $\mathcal{O}(\ln^{3/2}(n)/\sqrt{n})$ for the amount of resource constraints violations caused by the ergodic average of the population's iterates. This decay rate matches, up to the $\ln$-factor, with the fundamental limit of the convergence speed described by the lower complexity bound for black-box subgradient methods (see Theorem 3.2.1 in \cite{Nesterov2018}). Thus we expect that our result is optimal.

Another interesting occurrence we observed, both in the theoretical and numerical investigations, is that the choice of the mirror map might have an impact on the dimensional dependence of the quality of MAARP. This effect is rarely considered in the literature of the Nash equilibrium finding since it mostly uses the Euclidean projection as the mapping which realizes the first-order update in the feasible strategy set (to name a few: \cite{Paccagnan2017,Tatarenko2019,Pavel2016}). In future work, we aim to exploit this aspect further.   

%
\appendix
In this appendix, we relate the concept of the Nash equilibrium of the NGCC to another alternative concepts, which is more suitable for algorithmic analysis. The results stated in this appendix are known and can be found, for instance in \cite{Facchinei1}.
\subsection{Nash Equilibrium \& Variational Inequality}
\label{Subsec:aaojshddhdhjff}
The following concept is central to the analysis of first-order methods: 
\begin{definition}[Variational Inequality (VI)]
	\label{Def:VI}
	Let $\Z$ be a subset of a Euclidean normed space $\V$, and suppose that $\bg:\mathcal{Z}\rightarrow \V^{*}$. A point $\overline{\bx}\in\Z$ is a solution of the variational inequality $\text{VI}(\mathcal{Z},\bg)$, if $\left\langle \bx-\overline{\bx} ,\bg(\overline{\bx})\right\rangle\leq 0$, $\forall \bx\in\mathcal{Z}$.
	The set of solution of $\VI(\Z,\bg)$ is denoted by $\SOL(\Z,\bg)$.
\end{definition}
One helpful fact working with $\VI$ is that first-order methods tend to gain a drift toward the solution of it. For instance consider the iterate \eqref{Eq:aoaaosjsdhdjdhhddss} and $x_{*}\in\SOL(\X,\bv)$. If $\bv$ is monotone (which is fulfilled by Assumption \ref{Ass:aiishhfggfjhhdjjdhhdjjddd}), then it follows that:
\begin{equation}
\label{Eq:aaoaosjjshhdhdggshshsggssss}
\inn{\boldX_{k}-\bx_{*}}{\bv(\boldX_{k})}\leq \inn{\boldX_{k}-\bx_{*}}{\bv(\bx_{*}))}\leq 0,
\end{equation}
where the last inequality follows from the definition of $\VI$.
Thus the first-order feedback $v(\boldX_{k})$ for each time step $k+1$ forms an obtuse angle with the residual vector $X_{k}-\bx_{*}$, and consequently, $v(\boldX_{k})$ provides a direction toward $x_{*}$.

There is no burden working with $\SOL(\mathcal{Q},\bv)$ instead with $\GNE(\Gamma)$ since, as asserted in the following proposition, the solutions of $\VI(\mathcal{Q})$ is automatically a Nash equilibrium of $\Gamma$:
\begin{proposition}
	\label{Prop:aaoaojssjhhdjdhsss}
	If Assumption \ref{Ass:aiishhfggfjhhdjjdhhdjjddd} holds, then $\SOL (\mathcal{Q},\bv)\subseteq\GNE(\Gamma)$.
\end{proposition}
Proposition \ref{Prop:aaoaojssjhhdjdhsss} motivate us to highlight the generalized Nash equilibrium solving the corresponding variational inequality and call it as \textit{variational Nash equilibrium}.

Another nice thing about VI is that under mild condition one can establish existence of its solution:
\begin{proposition}
	\label{Prop:aoaosjsjshhdjdhdjd}
	Let $\mathcal{Z}$ be a non-empty subset of a Euclidean normed space $\V$ and $\bg:\Z\rightarrow \V^{*}$. 
	\begin{enumerate}
		\item If that $\Z$ is compact and convex, and if $\bg$ is continuous, then $\SOL(\Z,\bg)\neq \emptyset$.
		\item If $\SOL(\Z,\bg)\neq \emptyset$ and $\bg$ is strictly monotone on $\Z$, then $\SOL(\Z,\bg)$ is a singleton.
	\end{enumerate}
\end{proposition}

The proof of the first statement is based on the connection between $\VI$ and the fixed point of the Euclidean projection onto $\mathcal{Z}$. The latter can be described by means of Brouwer Fixed Point Theorem. The proof of second statement is closely related to the fact that if $\bg$ is strictly monotone and $\overline{\bx}\in\SOL(\Z,\bg)$, then:
\begin{equation}
\label{Eq:StrongMon}
\inn{\bx-\overline{\bx}}{\bg(\bx)}\leq\inn{\bx-\overline{\bx}}{\bg(\overline{\bx}}\leq 0,\quad\forall \bx\in\mathcal{Z},\bx\neq\overline{\bx}
\end{equation}
with equality if and only if $\bx=\overline{\bx}$.

From here, it is immediate to infer the existence of a Nash equilibrium. Indeed, this is obvious by setting $\bg=\bv$ and $\Z=\mathcal{Q}$ in Proposition	\ref{Prop:aoaosjsjshhdjdhdjd} and by noticing that the solution of $\VI$ is a Nash equilibrium (Proposition \ref{Prop:aaoaojssjhhdjdhsss}). 
\subsection{Decoupling the Coupled Constraints via of Lagrangian}
\label{Subsec:Decoup}
In order to investigate $\VI(\mathcal{Q},v)$, it is convenient to extend $\VI(\mathcal{Q},v)$ to $\VI(\mathcal{X}\times \real^{R}_{+},\tilde{\bv})$, where $\tilde{\bv}:\mathcal{X}\times \real^{R}_{+}$,
\begin{equation}
\label{Eq:iaaiisjsjdhhdhdhdhdjjsss}
\tilde{\bv}:\mathcal{X}\times \real^{R}_{+}\rightarrow \real^{D+R}, (\bx,\boldlambda)\mapsto \left[\bv(\bx)-[\nabla \bg(\bx)]^{\T}\boldlambda, \bg(\bx)\right]^{\T}.
\end{equation}
The advantage of this method is the decoupling of the constraint set. Specifically, employing this procedure, we only have to work with the constraint set $\X\times \real^{R}_{\geq 0}$ with product structure rather than with $\mathcal{Q}$. 

Usual KKT argumentation asserts that $\VI(\mathcal{Q},\bv)$ and $\VI(\mathcal{X}\times \real^{R}_{+},\tilde{\bv})$ are essentially the same in the following sense:
\begin{proposition}
	\label{Prop:aiaisshshjdhdggddd}
	Suppose that the Assumptions \ref{Ass:aiishhfggfjhhdjjdhhdjjddd}, \ref{Ass:aiishhfggfjhhdjjdhhdjjddd2} and \ref{Eq:aioaoosjjshhdjjddss} holds. Then the following statements are equivalent:
	\begin{enumerate}
		\item $\overline{\bx}\in\mathcal{Q}$ is a solution of $\VI(\mathcal{Q},\bv)$
		\item There exists $\overline{\boldlambda}\in\real^{R}_{\geq 0}$ s.t. $(\overline{\bx},\overline{\boldlambda})$ is a solution of $\VI(\mathcal{X}\times \real^{R}_{+},\tilde{\bv})$. 
	\end{enumerate}
\end{proposition}
\begin{proposition}
	\label{Prop:aiaisshshjdhdggddd3}
	Suppose that Assumptions \ref{Ass:aiishhfggfjhhdjjdhhdjjddd}, \ref{Ass:aiishhfggfjhhdjjdhhdjjddd2} and \ref{Eq:aioaoosjjshhdjjddss} hold. If $v$ is strictly monotone, then there exists a unique solution $\overline{\bx}$ of $\VI(\mathcal{Q},\bv)$ and $\overline{\boldlambda}\in\real^{R}_{\geq 0}$ s.t. $(\overline{\bx},\overline{\boldlambda})$ is a unique solution of $\VI(\mathcal{X}\times \real^{R}_{\geq 0},\tilde{\bv})$.
\end{proposition}

  \bibliographystyle{IEEEtran}
\bibliography{BibWard}

\begin{thebibliography}{10}
\providecommand{\url}[1]{#1}
\csname url@samestyle\endcsname
\providecommand{\newblock}{\relax}
\providecommand{\bibinfo}[2]{#2}
\providecommand{\BIBentrySTDinterwordspacing}{\spaceskip=0pt\relax}
\providecommand{\BIBentryALTinterwordstretchfactor}{4}
\providecommand{\BIBentryALTinterwordspacing}{\spaceskip=\fontdimen2\font plus
\BIBentryALTinterwordstretchfactor\fontdimen3\font minus
  \fontdimen4\font\relax}
\providecommand{\BIBforeignlanguage}[2]{{%
\expandafter\ifx\csname l@#1\endcsname\relax
\typeout{** WARNING: IEEEtran.bst: No hyphenation pattern has been}%
\typeout{** loaded for the language `#1'. Using the pattern for}%
\typeout{** the default language instead.}%
\else
\language=\csname l@#1\endcsname
\fi
#2}}
\providecommand{\BIBdecl}{\relax}
\BIBdecl

\bibitem{Mohsenian2010}
A.~Mohsenian-Rad, V.~W.~S. Wong, J.~Jatskevich, R.~Schober, and A.~Leon-Garcia,
  ``Autonomous {D}emand-{S}ide {M}anagement {B}ased on {G}ame-{T}heoretic
  {E}nergy {C}onsumption {S}cheduling for the {F}uture {S}mart {G}rid,''
  \emph{IEEE Trans. on Smart Grid}, vol.~1, no.~3, pp. 320 -- 331, Dec. 2010.

\bibitem{Saad2012}
W.~Saad, Z.~Han, H.~V. Poor, and T.~Basar, ``Game-theoretic methods for the
  smart grid: An overview of microgrid systems, demand-side management, and
  smart grid communications,'' \emph{IEEE Sig. Proc. Mag.}, vol.~29, no.~5, pp.
  86 -- 105, 2012.

\bibitem{Li2016}
S.~Li, W.~Zhang, J.~Lian, and K.~Kalsi, ``Market-{B}ased {C}oordination of
  {T}hermostatically {C}ontrolled {L}oads -- part {I}: {A} {M}echanism {D}esign
  {F}ormulation,'' \emph{IEEE Trans. on Pow. Sys.}, vol.~31, no.~2, pp. 1170 --
  1178, March 2016.

\bibitem{Ma2013}
Z.~Ma, D.~S. Callaway, and I.~A. Hiskens, ``Decentralized charging control of
  large populations of plug-in electric vehicles,'' \emph{IEEE Trans. on Cont.
  Sys. Tech.}, vol.~21, no.~1, pp. 67 -- 78, Jan. 2013.

\bibitem{Parise2014}
F.~Parise, M.~Colombino, S.~Grammatico, and J.~Lygeros, ``Mean field
  constrained charging policy for large populations of plug-in electric
  vehicles,'' in \emph{53rd IEEE Conf. on Dec. and Cont. (CDC)}, Dec. 2014, pp.
  5101 -- 5106.

\bibitem{Ma2016}
Z.~Ma, S.~Zou, L.~Ran, X.~Shi, and I.~A. Hiskens, ``Efficient decentralized
  coordination of large-scale plug-in electric vehicle charging,'' \emph{Aut.},
  vol.~69, pp. 35 -- 47, 2016.

\bibitem{Grammatico20162}
S.~Grammatico, ``Exponentially convergent decentralized charging control for
  large populations of plug-in electric vehicles,'' in \emph{IEEE 55th Conf. on
  Dec. and Cont. (CDC)}, Dec. 2016, pp. 5775 -- 5780.

\bibitem{Li2015}
N.~Li, L.~Chen, and M.~A. Dahleh, ``Demand {R}esponse {U}sing {L}inear {S}upply
  {F}unction {B}idding,'' \emph{IEEE Trans. on Smart Grid}, vol.~6, no.~4, pp.
  1827--1838, July 2015.

\bibitem{Barrera2015}
J.~Barrera and A.~Garcia, ``Dynamic {I}ncentives for {C}ongestion {C}ontrol,''
  \emph{IEEE Trans. on Aut. Cont.}, vol.~60, no.~2, pp. 299 -- 310, Feb. 2015.

\bibitem{Sandholm2012}
W.~H. Sandholm, \emph{Evolutionary Game Theory}.\hskip 1em plus 0.5em minus
  0.4em\relax New York, NY: Springer New York, 2012, pp. 1000--1029.

\bibitem{Newton2018}
J.~Newton, ``Evolutionary game theory: A renaissance,'' \emph{Games}, vol.~9,
  no.~2, 2018.

\bibitem{Nash1951}
J.~Nash, ``Non-cooperative games,'' \emph{Annals of Mathematics}, vol.~54,
  no.~2, pp. 286--295, 1951.

\bibitem{Shalev-Shwartz2012}
S.~Shalev-Shwartz, ``Online {L}earning and {O}nline {C}onvex {O}ptimization,''
  \emph{Foundations and Trends in Machine Learning}, vol.~4, no.~2, pp. 107 --
  194, 2012.

\bibitem{Bub2012}
S.~Bubeck and N.~Cesa-Bianchi, ``Regret {A}nalysis of {S}tochastic and
  {N}onstochastic {M}ulti-{A}rmed {B}andit {P}roblems,'' \emph{Found. and Tr.
  in Mach. Lear.}, vol.~5, no.~1, pp. 1--122, 2012.

\bibitem{Belmaga2018}
E.~V. Belmega, P.~Mertikopoulos, R.~Negrel, and L.~Sanguinetti, ``Online convex
  optimization and no-regret learning: Algorithms, guarantees and
  applications,'' \emph{arXiv:1804.04529}, 2018.

\bibitem{Nesterov2009}
Y.~Nesterov, ``Primal-dual subgradient methods for convex problems,''
  \emph{Math. Prog.}, vol. 120, no.~1, pp. 221--259, Aug. 2009.

\bibitem{Nemirovski2008}
A.~Nemirovski, A.~Juditsky, G.~Lan, and A.~. Shapiro, ``Robust {S}tochastic
  {A}pproximation {A}pproach to {S}tochastic {P}rogramming,'' \emph{SIAM J. on
  Opt.}, vol.~19, no.~4, pp. 1574 -- 1609, Jan. 2008.

\bibitem{Chen2018}
T.~{Chen}, Q.~{Ling}, Y.~{Shen}, and G.~B. {Giannakis}, ``{H}eterogeneous
  {O}nline {L}earning for {Thing-Adaptive} {F}og {C}omputing in {IoT},''
  \emph{IEEE Internet of Things Journal}, vol.~5, no.~6, pp. 4328 -- 4341, Dec.
  2018.

\bibitem{Scutari2012}
G.~Scutari, D.~P. Palomar, F.~Facchinei, and J.-S. Pang, \emph{Monotone Games
  for Cognitive Radio Systems}.\hskip 1em plus 0.5em minus 0.4em\relax London:
  Springer, 2012, pp. 83--112.

\bibitem{Abbas2016}
G.~{Abbas}, Z.~{Halim}, and Z.~H. {Abbas}, ``Fairness-driven queue management:
  A survey and taxonomy,'' \emph{IEEE Communications Surveys Tutorials},
  vol.~18, no.~1, pp. 324 -- 367, Firstquarter 2016.

\bibitem{Borch2016}
C.~Borch, ``High-frequency {T}rading, {A}lgorithmic {F}inance, and the {F}lash
  {C}rash: {R}eflections on {E}ventalization,'' \emph{Economy and Society},
  vol.~45, no. 3 -- 4, pp. 350 -- 378, Jan. 2016.

\bibitem{Fudenberg1998}
D.~Fudenberg, \emph{The Theory of Learning in Games}.\hskip 1em plus 0.5em
  minus 0.4em\relax MIT Press, 1998.

\bibitem{Zhang2019}
K.~Zhang, Z.~Yang, and T.~Ba\c{s}ar, ``{M}ulti-{A}gent {R}einforcement
  {L}earning: {A} {S}elective {O}verview of {T}heories and {A}lgorithms,''
  \emph{arXiv:1911.10635}, 2019.

\bibitem{Mertikopoulos2018}
P.~Mertikopoulos and Z.~Zhou, ``Learning in games with continuous action sets
  and unknown payoff functions,'' \emph{Math. Prog.}, Mar. 2018.

\bibitem{DuvocellaMertikopoulos2018}
B.~Duvocelle, P.~Mertikopoulos, M.~Staudigl, and D.~Vermeulen, ``Learning in
  time-varying games,'' \emph{CoRR}, vol. abs/1809.03066, 2018.

\bibitem{ZhouMertikopoulos2018}
\BIBentryALTinterwordspacing
Z.~Zhou, P.~Mertikopoulos, N.~Bambos, P.~Glynn, and C.~Tomlin, ``Multi-agent
  online learning with imperfect information,'' \emph{Submitted.}, 2018.
  [Online]. Available:
  \url{http://mescal.imag.fr/membres/panayotis.mertikopoulos/files/Delayed.pdf}
\BIBentrySTDinterwordspacing

\bibitem{Facchinei1}
F.~Facchinei and J.-S. Pang, \emph{Finite-Dimensional Variational Inequalities
  and Complementarity Problems}.\hskip 1em plus 0.5em minus 0.4em\relax
  Springer-Verlag New York, 2003.

\bibitem{Facchinei2007}
F.~Facchinei and C.~Kanzow, ``Generalized {N}ash equilibrium problems,''
  \emph{4OR}, vol.~5, no.~3, pp. 173--210, Sep. 2007.

\bibitem{Pavel2007}
L.~Pavel, ``An {E}xtension of {D}uality to a {G}ame-theoretic {F}ramework,''
  \emph{Automatica}, vol.~43, no.~2, pp. 226--237, Feb. 2007.

\bibitem{Yin2011}
H.~{Yin}, U.~V. {Shanbhag}, and P.~G. {Mehta}, ``{N}ash {E}quilibrium
  {P}roblems {W}ith {S}caled {C}ongestion {C}osts and {S}hared {C}onstraints,''
  \emph{IEEE Transactions on Automatic Control}, vol.~56, no.~7, pp.
  1702--1708, 2011.

\bibitem{Kulkarni2012}
A.~A. Kulkarni and U.~V. Shanbhag, ``On the variational equilibrium as a
  refinement of the generalized {N}ash equilibrium,'' \emph{Automatica},
  vol.~48, no.~1, pp. 45 -- 55, 2012.

\bibitem{Arslan2012}
G.~Arslan, M.~F. Demirkol, and S.~Y{\"u}ksel, ``On {G}ames with {C}oupled
  {C}onstraints,'' in \emph{51st IEEE CDC}, 2012, pp. 7151--7156.

\bibitem{Johari2004}
R.~Johari, ``Efficiency loss in market mechanisms for resource allocation,''
  Ph.D. dissertation, MIT, 2004.

\bibitem{Roughgarden2004}
T.~Roughgarden and {\'E}.~Tardos, ``Bounding the inefficiency of equilibria in
  nonatomic congestion games,'' \emph{Games and Economic Behavior}, vol.~47,
  no.~2, pp. 389 --203, 2004.

\bibitem{Papadimitriou2001}
C.~Papadimitriou, ``Algorithms, games, and the internet,'' in \emph{Proc. of
  33th ACM on the Theory of Computing}, 2001, pp. 749 -- 753.

\bibitem{Kulkarni2019}
A.~A. Kulkarni, ``The {E}fficiency of {G}eneralized {N}ash and {V}ariational
  {E}quilibria,'' \emph{arXiv:1908.00702}, 2019.

\bibitem{Paccagnan2017}
D.~Paccagnan, B.~Gentile, F.~Parise, M.~Kamgarpour, and J.Lygeros, ``Nash and
  {W}ardrop {E}quilibria in {A}ggregative {G}ames with {C}oupling
  {C}onstraints,'' \emph{IEEE Trans. on Aut. Cont.}, 2017.

\bibitem{Grammatico2017}
S.~{Grammatico}, ``Dynamic {C}ontrol of {A}gents {P}laying {A}ggregative
  {G}ames with {C}oupling {C}onstraints,'' \emph{IEEE Transactions on Automatic
  Control}, vol.~62, no.~9, pp. 4537--4548, Sep. 2017.

\bibitem{Marden2009}
J.~R. Marden, H.~P. Young, G.~Arslan, and J.~S. Shamma, ``Payoff-based dynamics
  for multiplayer weakly acyclic games,'' \emph{SIAM Journal on Control and
  Optimization}, vol.~48, no.~1, pp. 373--396, 2009.

\bibitem{Tatarenko2019}
T.~{Tatarenko} and M.~{Kamgarpour}, ``Learning generalized nash equilibria in a
  class of convex games,'' \emph{IEEE Transactions on Automatic Control},
  vol.~64, no.~4, pp. 1426--1439, Apr. 2019.

\bibitem{Marden2013}
J.~R. {Marden}, S.~D. {Ruben}, and L.~Y. {Pao}, ``{A} {M}odel-{F}ree {A}pproach
  to {W}ind {F}arm {C}ontrol {U}sing {G}ame {T}heoretic {M}ethods,'' \emph{IEEE
  Transactions on Control Systems Technology}, vol.~21, no.~4, pp. 1207 --
  1214, 2013.

\bibitem{Zhu2013}
M.~Zhu and S.~Mart{\'i}nez, ``Distributed {C}overage {G}ames for
  {E}nergy-{A}ware {M}obile {S}ensor {N}etworks,'' \emph{SIAM Journal on
  Control and Optimization}, vol.~51, no.~1, pp. 1--27, 2013.

\bibitem{Srikant2004}
R.~Srikant, \emph{The Mathematics of Internet Congestion Control}.\hskip 1em
  plus 0.5em minus 0.4em\relax Birkhaeuser, 2004.

\bibitem{Low1999}
S.~H. Low and D.~E. Lapsley, ``Optimization flow control. {I}. {B}asic
  algorithm and convergence,'' \emph{IEEE/ACM Trans. on Net.}, vol.~7, no.~6,
  pp. 861--874, Dec. 1999.

\bibitem{Paganini2005}
F.~{Paganini}, Z.~{Wang}, J.~C. {Doyle}, and S.~H. {Low}, ``Congestion control
  for high performance, stability, and fairness in general networks,''
  \emph{IEEE/ACM Transactions on Networking}, vol.~13, no.~1, pp. 43--56, 2005.

\bibitem{Rockafellar1970}
R.~T. Rockafellar, \emph{Convex Analysis}.\hskip 1em plus 0.5em minus
  0.4em\relax Princeton University Press, 1970.

\bibitem{Rockafellar1998}
R.~T. Rockafellar and R.~J.~B. Wets, \emph{Variational Analysis}, ser. A Ser.
  of Comp. Stud. in Math.\hskip 1em plus 0.5em minus 0.4em\relax
  Springer-Verlag, 1998, vol. 317.

\bibitem{Mertikopoulos2016}
P.~Mertikopoulos and W.~H. Sandholm., ``Learning in games via reinforcement and
  regularization,'' \emph{Math. of Op. Res.}, vol.~14, no.~1, pp. 124 -- 143,
  2016.

\bibitem{Chen1993}
G.~Chen and M.~Teboulle, ``Convergence {A}nalysis of a {P}roximal-{L}ike
  {M}inimization {A}lgorithm {U}sing {B}regman {F}unctions,'' \emph{SIAM
  Journal on Optimization}, vol.~3, no.~3, pp. 538--543, 1993.

\bibitem{Mahdavi1}
M.~Mahdavi, R.~Jin, and T.~Yang, ``Trading {R}egret for {E}fficiency: {O}nline
  {C}onvex {O}ptimization with {L}ong {T}erm {C}onstraints,'' \emph{J. Mach.
  Learn. Res.}, vol.~13, no.~1, pp. 2503 -- 2528, Jan. 2012.

\bibitem{Benaim1999}
M.~Bena{\"i}m, ``Dynamics of stochastic approximation algorithms,'' in
  \emph{S{\'e}minaire de Probabilit{\'e}s XXXIII}, J.~Az{\'e}ma, M.~{\'E}mery,
  M.~Ledoux, and M.~Yor, Eds.\hskip 1em plus 0.5em minus 0.4em\relax Berlin,
  Heidelberg: Springer Berlin Heidelberg, 1999, pp. 1--68.

\bibitem{Bharath1999}
B.~Bharath and V.~S. Borkar, ``Stochastic approximation algorithms: Overview
  and recent trends,'' \emph{Sadhana}, vol.~24, pp. 425 -- 452, 1999.

\bibitem{Borkar2008}
V.~S. Borkar, \emph{Stochastic {A}pproximation: {A} {D}ynamical {S}ystems
  {V}iewpoint}.\hskip 1em plus 0.5em minus 0.4em\relax Hindustan Book Agency,
  2008, vol.~48.

\bibitem{Hall1980}
P.~Hall and C.~Heyde, \emph{Martingale Limit Theory and its Application}, ser.
  Probability and Mathematical Statistics: A Series of Monographs and
  Textbooks.\hskip 1em plus 0.5em minus 0.4em\relax Academic Press, 1980.

\bibitem{Duflo1997}
M.~Duflo, \emph{Random {I}terative {M}odels}, ser. 1st.\hskip 1em plus 0.5em
  minus 0.4em\relax Berlin, Heidelberg: Springer-Verlag, 1997.

\bibitem{Boucheron2013}
S.~Boucheron, G.~Lugosi, and P.~Massart, \emph{Concentration inequalities: {A}
  nonasymptotic theory of independence}.\hskip 1em plus 0.5em minus 0.4em\relax
  Oxford University Press, 2013.

\bibitem{Nesterov2018}
Y.~Nesterov, \emph{Lectures on Convex Optimization}.\hskip 1em plus 0.5em minus
  0.4em\relax Springer International Publishing, 2018.

\bibitem{Pavel2016}
F.~Salehisadaghiani and L.~Pavel, ``Distributed {N}ash equilibrium seeking: A
  gossip-based algorithm,'' \emph{Automatica}, vol.~72, pp. 209 -- 216, 2016.

\end{thebibliography}
\vskip 0pt plus -1fil
\begin{IEEEbiography}[{\includegraphics[width=1in,height=1.25in,clip,keepaspectratio]{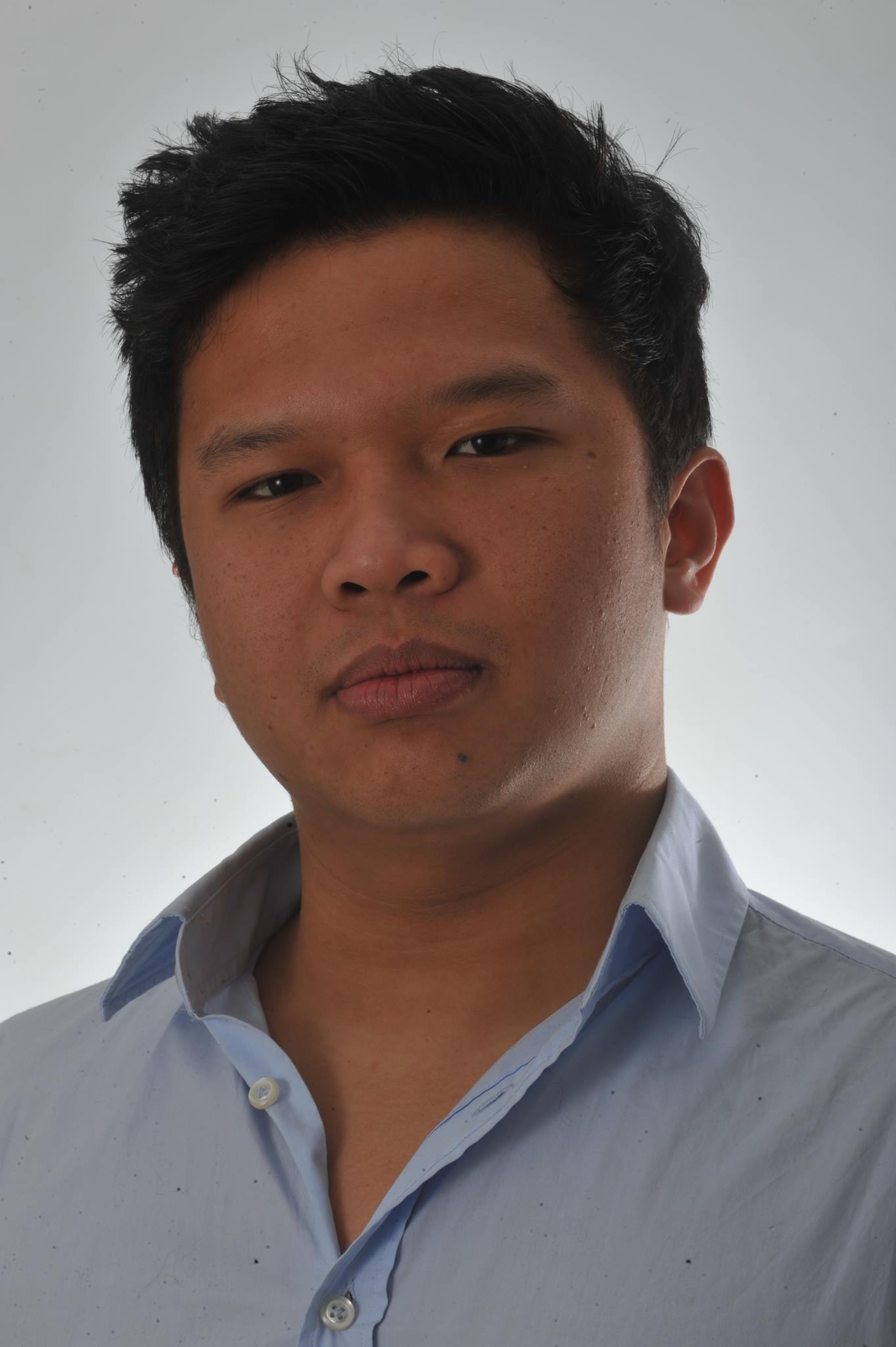}}]{Ezra Tampubolon}
	(S'16) is currently a Ph.D. candidate in the Department of Electrical and Computer Engineering at the Technische Universit{\"a}t M{\"u}nchen, Munich, Germany. He received the B.Sc. degree in 2013 and the M.Sc. degree in 2015 in Electrical and Computer Engineering from the Technische Universit{\"a}t M{\"u}nchen. His research interests lie in the area of stochastic optimization, distributed algorithms, and their applications to the machine learning, signal processing and communications.   
\end{IEEEbiography}
\vskip 0pt plus -1fil
\begin{IEEEbiography}[{\includegraphics[width=1in,height=1.25in,clip,keepaspectratio]{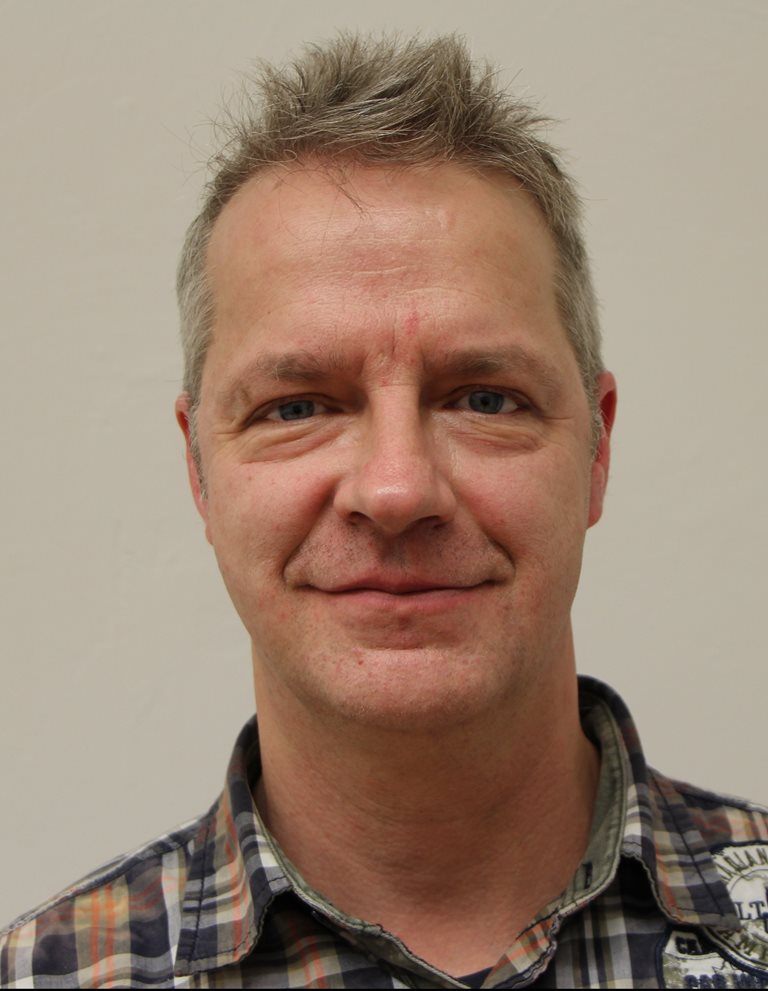}}]{Holger
Boche}
(M'04 - SM'07 - F'11) received the Dipl.-Ing.\ and Dr.-Ing.\ degrees in electrical engineering from the Technische Universit{\"a}t Dresden, Dresden, Germany, in 1990 and 1994, respectively. He graduated in mathematics from the Technische Universit{\"a}t Dresden in 1992. From
1994 to 1997, he did Postgraduate studies in mathematics at the Friedrich-Schiller Universit{\"a}t Jena, Jena, Germany. He received his Dr. rer. nat. degree in pure mathematics from the Technische Universit{\"a}t Berlin, Berlin, Germany, in 1998. In 1997, he joined the Heinrich-Hertz-Institut (HHI) f{\"u}r Nachrichtentechnik Berlin, Berlin, Germany. Starting in 2002, he was a Full Professor for mobile communication networks with the Institute for Communications Systems, Technische Universit{\"a}t Berlin. In 2003, he became Director of the Fraunhofer German-Sino Lab for Mobile Communications, Berlin, Germany, and in 2004 he became the Director of the Fraunhofer Institute for Telecommunications (HHI), Berlin, Germany. Since October 2010 he has been with the Institute of Theoretical Information Technology and Full Professor at the Technische Universit{\"a}t M{\"u}nchen, Munich, Germany. Since 2014 he has been a member and honorary fellow of the TUM Institute for Advanced Study, Munich, Germany. He was a Visiting Professor with the ETH Zurich, Zurich, Switzerland, during the 2004 and
2006 Winter terms, and with KTH Stockholm, Stockholm, Sweden, during the
2005 Summer term. Prof. Boche is a Member of IEEE Signal Processing Society SPCOM and SPTM Technical Committee. He was elected a Member of the German Academy of Sciences (Leopoldina) in 2008 and of the Berlin Brandenburg Academy of Sciences and Humanities in 2009. He received the Research Award ``Technische Kommunikation'' from the Alcatel SEL Foundation in October 2003, the ``Innovation Award'' from the Vodafone Foundation in June 2006, and the Gottfried Wilhelm Leibniz Prize from the Deutsche Forschungsgemeinschaft (German Research Foundation) in 2008. He was co-recipient of the 2006 IEEE Signal Processing Society Best Paper Award and recipient of the 2007 IEEE Signal Processing Society Best Paper Award. He was the General Chair of the Symposium on Information Theoretic Approaches to Security and Privacy at IEEE GlobalSIP 2016. Among his publications is the recent book Information Theoretic Security and Privacy of Information Systems (Cambridge University Press).
\end{IEEEbiography}

\end{document}